\documentclass[journal]{IEEEtran}
\usepackage[latin1]{inputenc}
\usepackage[english]{babel}
\usepackage{amsmath,amsfonts,amssymb,bm,bbm}
\usepackage{amsthm}
\usepackage{subfigure}
\usepackage{textcomp, color}
\usepackage{hyperref}
\usepackage{emptypage}  
\usepackage{lipsum} 
\usepackage{microtype} 
\usepackage{graphicx,float}
\usepackage{tikz}
\usepackage{graphics}

\newcommand{\ds}{\displaystyle}

\newcommand{\ba}{\begin{array}}
\newcommand{\ea}{\end{array}}

\newcommand{\1}{\mathbbm{1}}

\newcommand{\be}{\begin{equation}}
\newcommand{\ee}{\end{equation}}

\newcommand{\ups}{\upsilon}

\newcommand{\mc}{\mathcal}

\newcommand{\ov}{\overline}

\newcommand{\R}{\mathbb{R}}

\newcommand{\de}{\mathrm{d}}

\newcommand{\se}{\text{ if }}

\DeclareMathOperator*{\argmin}{argmin}

\def\R{\mathbb{R}}

\newtheorem{theorem}{Theorem}
\newtheorem{assumption}{Assumption}
\newtheorem{corollary}{Corollary}
\newtheorem{definition}{Definition}
\newtheorem{lemma}{Lemma}
\newtheorem{proposition}{Proposition}
\newtheorem{remark}{Remark}
\newtheorem{example}{Example}
\usepackage{color,hyperref} 
\newcommand{\CorrR}[1]{{\color{black} #1}}         
\newcommand{\CorrB}[1]{{\color{black} #1}}

\ifCLASSINFOpdf
\else
\fi
\begin{document}
%
\title{Distributed Dynamic Pricing of\\ Multiscale Transportation Networks}
%
%
%

\author{Giacomo~Como,~\IEEEmembership{Member,~IEEE,}
        and~Rosario~Maggistro
\thanks{G. Como is with the Department of Mathematical Sciences, Politecnico di Torino, Corso Duca degli Abruzzi 24, 10129, Torino, Italy.  Email: giacomo.como@polito.it. He is also affiliated with the Department of Automatic Control, Lund University, Sweden.}
\thanks{R. Maggistro was with the Department of Mathematical Sciences, Politecnico di Torino, Italy. He is now with the Department of Management, Universit\`{a} Ca' Foscari Venezia, Italy. Email: rosario.maggistro@unive.it}
\thanks{ \CorrR{A preliminary version of this paper appeared in part as \cite{CDC2018}.} This reasearch was carried on within the framework of the MIUR-funded {\it Progetto di Eccellenza} of the {\it Dipartimento di Scienze Matematiche G.L.~Lagrange}, CUP: E11G18000350001, and was partly supported by the {\it Compagnia di San Paolo} and the Swedish Research Council.}
}

\maketitle

\begin{abstract}
We study transportation networks controlled by dynamic feedback 
tolls. We focus on a multiscale model whereby the dynamics of the traffic flows are intertwined with those of the routing choices. The latter are influenced by the current traffic state of the network as well as by dynamic tolls controlled in feedback by the system planner. We prove that a class of decentralized monotone \CorrR{flow-dependent tolls} allow for globally stabilizing the transportation network around a generalized Wardrop equilibrium. In particular, our results imply that using decentralized marginal cost tolls, stability of the dynamic transportation network is guaranteed around the social optimum traffic assignment. This is particularly remarkable as such dynamic feedback tolls can be computed in a fully local way without the need for any global information about the network structure, its state, or the exogenous network loads. Through numerical simulations, we also compare the performance of such decentralized dynamic feedback marginal cost tolls with constant off-line (and centrally) optimized tolls both in the asymptotic and in the transient regime and we investigate their robustness to information delays. 
\end{abstract}


%
\IEEEpeerreviewmaketitle

\section{Introduction}
Over the past years there has been an increasing interest in the control analysis and synthesis of dynamical transportation networks. This is especially motivated by the wide-spreading sensing, communication, information, and actuation technologies that are dramatically changing the transportation system dynamics and affecting the users' decision making and behavior. There is a growing awareness that the new opportunities and risks created by these technologies can be fully understood only within a dynamical network framework. 
 
 {Dynamics and control of traffic flows over networks have received a great deal of research attention, motivated by applications both to communication networks \cite{BertsekasGallager:1992}--\cite{Srikant:2004}  and to road transportation systems \cite{HegyiSchutter1}--\cite{Nilsson.Como:2019}. 
Special emphasis in this literature has been put on mathematical properties of the dynamical system model  ---e.g., convexity, monotonicity, contractivity, Lyapunov functions' separability--- that allow for scalable control architectures such as, e.g., distributed or decentralized control policies \cite{Chiang.ea:2007}--\cite{Como}. 


A central aspect of dynamical flow networks is related to the routing decisions. In classical approaches to road traffic networks, the routing is considered static (see, e.g., the Cell Transmission Model \cite{Daganzo}), possibly determined by a network flow optimization problem such as a system or user optimum traffic assignment problem (\cite{Ahuja,Patriksson}). 
\CorrR{In fact, it is widely recognized that when drivers make their routing decisions by choosing the paths that minimize their own experienced delays, network congestion can increase significantly with respect to a hypothetical scenario where a central planner was able to directly impose an optimized routing, a phenomenon known as the price of anarchy \cite{Roughgarden,Brown}.}
On the other hand, the impact of  dynamic routing on the stability and resilience of traffic flow networks has been recently analysed
\cite{Robust1}--\cite{ComoSavla} 
and there has been also a significant research effort to understand the drivers' answer to external communications from intelligent traveller information devices 
\cite{Mahamassani}--\cite{Amin}. }
{Charging tolls or providing signalling schemes subject to a non-trivial amount of uncertainty are, therefore, two potential strategies to influence drivers to make routing choices that result in globally optimal routing (see \cite{Smith}--\cite{Marecek2}). 

In this paper, we study multiscale dynamical flow networks whereby the physical dynamics of the traffic flows are intertwined with those of the routing choices. 
}In particular, we extend the model and results of \cite{ComoSavla} by introducing decentralized \CorrR{flow-dependent tolls} in order to influence the route choice behavior. Specifically, we consider a multiscale dynamical model of the transportation network whereby the traffic dynamics describing the real time evolution of the local {traffic} level are coupled with those of the  path preferences. We assume that the latter evolve following a perturbed best response to global information about the traffic status of the whole network and to decentralized \CorrR{flow-dependent tolls}.

Our main result shows that by using monotone decentralized \CorrR{flow-dependent tolls} and in the limit of small update rate of the aggregate path preferences, the transportation network  globally stabilizes around a generalized Wardrop equilibrium \cite{Wardrop}. The latter is a configuration in which the
perceived cost associated to any source-destination path chosen by a nonzero
fraction of users does not exceed the perceived cost associated to any other
path. As in \cite{ComoSavla}, we assume that the path preferences evolve at a slower time scale than the physical traffic flows and adopt a singular perturbation approach \cite{Khalil} to the stability analysis of the ensuing multiscale closed-loop traffic dynamics. In fact, classical results from evolutionary game theory and population dynamics \cite{Hofbauer}--\cite{SandholmLibro} cannot be directly applied to our framework since they assume that information is accessed at a single temporal and spatial scale while the traffic dynamics are neglected as they are assumed to be instantaneously equilibrated.

The introduction of tolls has long been studied as a way to influence the rational and selfish behavior of drivers so that the associated user equilibrium can be aligned with the system optimum network flow. A particular taxation mechanism that guarantees this alignment is marginal-cost pricing, see, e.g., \cite{Beckmann} and \cite{Sandholm}. Marginal-cost tolls do not require any global information about the network structure or traffic state, nor of the exogenous user demands, and can be computed in a fully local way.  We prove that, using marginal-cost tolls our multiscale model of dynamical flow network stabilizes around the social optimum traffic assignment. It is worth observing that our results go well beyond the traditional setting \cite{Beckmann} where only static frameworks are considered as well as the evolutionary game theoretic approaches \cite{Sandholm} where only path preference dynamics are consider, neglecting the physical ones that are assumed equilibrated. In fact, our analysis is carried over in a fully dynamical flow network setting. In this respect, \CorrB{the global optimality guarantees obtained in this paper should be compared with other recent results on global performance and resilience of  robust distributed control of dynamical flow networks \cite{Robust1}, \cite{Yazicioglu}.}

\CorrR{In the last part of the paper, we present numerical simulations comparing the asymptotic and transient performance of the system with dynamic distributed  feedback marginal cost tolls and constant marginal cost tolls. While it is known that the latter can be computed to enforce the social optimum equilibrium  provided that the system planner has a complete knowledge of the network topology, user demand profile, and delay functions,   
we show that not only do the former achieve the same optimal asymptotic performance but they also guarantee faster convergence and are strongly robust to variation of network topology and exogenous traffic load.}
\CorrR{It is worth pointing out that robustness of the marginal cost tolls was recently investigated also in the case of static models \cite{Brown}, \cite{BrownTAC}}.
Finally, we study the effect of time-delays in the global information of the routing decision dynamics dynamics and analyze their influence on the evolution of the multi-scale dynamical system. For different values of such time delays, one observes different behaviors of the system depending on whether dynamic feedback marginal cost tolls are used instead of constant marginal cost ones. With the latter, the system remains stable and converges to the equilibrium, instead with the former a phase transition and an oscillatory behavior may emerge as the for large enough delays.

The rest of this paper is organized as follows. In Section \ref{section2}, we describe the multiscale model of network traffic flow dynamics and introduce distributed dynamics tolls. In Section \ref{section3} we state and discuss the main technical results of the paper, whose proofs are then presented in Section \ref{section4}. In Section \ref{sect:extensions} we discuss possible extensions of the results presented in the previous sections.
In Section \ref{section5} we provide a numerical study of the transient and asymptotic performance of both dynamic feedback and constant tolls and also analyze their robustness with respect to information delays.
Section \ref{section6} draws conclusions and suggests future works.
\subsection{Notation}
Let $\mathcal A$ and $\mathcal B$ be finite
sets. Then $\vert \mathcal A\vert$ denotes the cardinality of $\mathcal A$, $\mathbb{R}^{\mathcal A}$ the space
of real-valued vectors whose components are indexed by
elements of $\mathcal A$, and $\mathbb{R}^{\mathcal A\times \mathcal B}$ the space of real-valued matrices whose entries are indexed by pairs in $\mathcal A\times \mathcal B$. The transpose of a matrix $Q $ in $ \mathbb{R}^{\mathcal A\times \mathcal B}$ is denoted by $Q' $ in $ \mathbb{R}^{\mathcal B\times \mathcal A}$, $I$ is an identity matrix and $\mathbf{1}$ an all-one vector whose size depends on the context. {For, $i$ in $\mc A$, $\delta^{(i)}$ in $\R^{\mc A}$ denotes the vector with all entries equal to $0$ except for the $i$-th that is equal to $1$.} We use the notation $\Phi:=I-\vert \mathcal A\vert^{-1}\mathbf{11'} $ in $ \mathbb{R}^{\mathcal A\times \mathcal A}$ to denote the projection matrix of the space orthogonal to $\mathbf{1}$.
The simplex of a probability vector over $\mathcal A$ is denoted by $S(\mathcal A)=\{x \in \mathbb{R}_{+}^\mathcal{A} : \mathbf{1}'x=1\}$.
Let $\Vert \cdot\Vert_p$ be the class of $p$-norms for $p$ in $ [1, \infty]$, and by default, let $ \Vert \cdot \Vert:=\Vert \cdot \Vert_2$. Let now $\text{sgn}:\mathbb{R}\to \{-1, 0, 1\}$ be the sign function, defined by $\text{sgn}(x)=1$ if $x>0$, $\text{sgn}(x)=-1$ if $x<0$ and $\text{sgn}(x)=0$ if $x=0$. 
By convention, we will assume the identity $d\vert x\vert/dx=\text{sgn}(x)$ to be valid for every $x $ in $ \mathbb{R}$, including $x=0$. Finally, given the gradient $\nabla f$ of a function $f:D\to \mathbb{R}$ with $D\subseteq\mathbb{R}^{\mathcal A} $, we denote with  $\tilde{\nabla} f=\Phi\nabla f$  the projected gradient on $S(\mathcal A)$.
\section{Model description}\label{section2}

\subsection{Transportation network}
We model the topology of the transportation network as a directed multi-graph $\mathcal{G=(V, E)}$, where $\mathcal{V}$ is a finite set of nodes and $\mathcal{E}$ is a finite set of directed links. Each link $i$ in $\mc E$ is directed from its tail node $\theta_i$ to its head node $\kappa_i\neq \theta_i$. \CorrR{We shall allow for parallel links, i.e., links $i\neq j$ such that $\theta_i=\theta_j$ and $\kappa_i=\kappa_j$, hence the prefix in multi-graph}. On the other hand, we shall assume that there are no self-loops, i.e., that $\theta_i\ne\kappa_i$ for every link $i$ in $\mc E$. 
We shall denote by $B$ in $\{-1,0,1\}^{\mathcal V\times\mathcal E}$  the node-link incidence matrix of a multigraph $\mathcal G$, whose entries are given by
\begin{equation*}
	B_{vi}=
	\begin{cases}
		+1 & \text{if} \quad v=\theta_i\\
		-1  & \text{if} \quad v=\kappa_i\\
		0 & \text{if} \quad v\neq\theta_i, \kappa_i. 
	\end{cases}
\end{equation*}
A length-$l$ path from a node $v_0$ to a node $v_{l}$ is an ordered $l$-tuple of links $\gamma=(i_1, i_2, \ldots, i_{l})$ such that the tail node of the first link is $\theta_{i_1}=v_0$, the head node of the last link is $\kappa_{i_l}=v_{l}$, the tail node of the next link coincides with the head node of the previous link, i.e., $v_s=\kappa_{i_{s}}=\theta_{i_{s+1}}$ for $1\le s\le l-1$, and no node is visited twice, i.e., $v_{r}\ne v_s$ for all $0\le r<s\le l$, except possibly for $v_0=v_l$, in which case the path is referred to a cycle.  A node $d$ is said to be reachable from another node $o$ if there exists at least a path from $o$ to $d$. Observe that, in contrast to \cite{ComoSavla} where the transportation network was assumed to be cycle-free, in this paper we allow for the possible presence of cycles. 

Throughout the paper, we will consider a given origin node $o$ and a destination node $d\ne o$ that is reachable from $o$ and let $\Gamma$ be the set of paths from $o$ to $d$ of any length $l\ge1$. We shall denote the corresponding link-path incidence matrix by $A$ in $\{0,1\}^{\mathcal E\times\Gamma}$ with entries
\begin{equation*}
	A_{i\gamma}=
	\begin{cases}
		1 & \text{if} \quad i \in \gamma, \\
		0  & \text{if} \quad i\notin \gamma.
	\end{cases}
\end{equation*} 
We shall assume that every link $i$ lies on some path from $o$ to $d$ so that $A$ has no all-zero rows. 
We shall refer to nonnegative vectors $y$ in $\R_+^{\mc E}$ generally as flow vectors. 
\CorrR{Upon recalling that $\delta^{(o)}$ ($\delta^{(d)}$) is the vector with all entries equal to $0$ except for the one in the origin (destination) node that is equal to $1$,} we shall refer to a flow vector $y$ such that 
\be\label{equilibrium-flow}By=\lambda\left(\delta^{(o)}-\delta^{(d)}\right)\,,\ee
for some $\lambda\ge0$  as an  $o$-$d$ equilibrium flow vector of throughput $\lambda$.
For $\lambda\ge0$, let us consider the simplex
\be\label{def:Slambda}\mc S_{\lambda}=\left\{z\in\R_+^{\Gamma}:\,\1'z=\lambda\right\}\,.\ee
Observe that, for every $z$ in $\mc S_{\lambda}$, one has $BAz=\lambda(\delta^{(o)}-\delta^{(d)})$, so that 
\be\label{yz}y^z:=Az\ee
is an $o$-$d$ equilibrium flow vector of throughput $\lambda$. Throughout, we shall refer to any $z$ in $\mc S_{\lambda}$ as a \emph{path preference} vector and to $y^z$ defined as in \eqref{yz} as the \emph{associated equilibrium flow} vector.

Each link $i$ in $\mc E$ of the transportation network topology $\mc G$ represents a cell. \CorrR{We shall denote the density on and the outflow from cell $i$ in $ \mathcal{E}$ by $x_i$ and $y_i$, respectively.}
We shall assume that density and outflow of each cell are related by a functional dependence
\begin{equation}
	y_i=\varphi_i(x_i), \qquad i \in \mathcal{E},
\end{equation}
satisfying the following property.

\begin{assumption}\label{assumption:flow-function} 
	For every link $i $ in $ \mathcal{E}$ the flow-density function $\varphi_i:\mathbb{R}_+\to \mathbb{R}_+$ is twice continuously differentiable, strictly increasing, strictly concave, and such that
	$$\varphi_i(0)=0, \qquad \varphi_i'(0)< +\infty\,.$$
\end{assumption}
For every link $i$ in $\mathcal{E}$, let
$$
C_i:=\sup\{\varphi_i(x_i): x_i\geq 0\}
$$
be its maximum flow capacity.
{
\begin{remark}\label{remark:increasing}
Notice that in road traffic networks the assumption that the flow-density functions are strictly increasing remains valid provided that we confine ourselves to the free-flow region, as is done in \cite{ComoSavla}. In Section \ref{sect:extensions} we will discuss how the framework of this paper could possibly be extended to more accurate dynamical models for road traffic flow networks, such as the Cell Transmission Model \cite{Daganzo}. 
\end{remark}}
Let us denote cell $i$'s latency function by $\tau_i:\R_+\to[0,+\infty]$. Such latency function returns the delay incurred in traversing link $i$ in $\mathcal E$, when the current flow out of it is $y_i$, and it is defined by 
\begin{equation}\label{delayfunction}
	\tau_i(y_i):=
	\begin{cases}
		\displaystyle{1}/{\varphi'_i(0)} \quad &\text{if}\ y_i=0\\[7pt]
		\displaystyle{\varphi_i^{-1}(y_i)}/{y_i} &\text{if}\ 0<y_i<C_i  \\[7pt]
		\displaystyle +\infty &\text{if}\ y_i\geq C_i\,. 
	\end{cases}
\end{equation}
\CorrR{Notice that the third line in \eqref{delayfunction} is merely a convenient mathematical convention allowing  us to formally extend the  range of the flow variable $y_i$ to values above the cell $i$'s capacity, albeit such values of flow remain not physically achievable.}
The following simple but useful result is proven in Appendix \ref{proof:lemma-convexity}.
\begin{lemma}\label{lemma:convexity}
	Let $\varphi_i:\R_+\to\R_+$ be a flow-density function satisfying Assumption \ref{assumption:flow-function}. Then, the corresponding latency function $\tau_i$ defined in \eqref{delayfunction} is twice continuously differentiable, strictly increasing on the interval $[0,C_i)$, and such that $\tau_i(0) > 0$. Moreover, its first derivative is given by 
	\be\label{tau'} \tau_i'(y)=\frac{y-x\varphi_i'(x)}{\varphi_i'(x)y^2}\,,\qquad x=\varphi_i^{-1}(y)\,,\ee
	and the function $y\mapsto y\tau_i(y)$ is strictly convex on $[0,C_i)$. 
\end{lemma}

Let us now define the set of \emph{feasible flow} vectors as $$\mc F:=\left\{y\in\R_+^{\mc E}:\,y_i< C_i\,,\ i\in\mc E\right\}$$ and the set of \emph{feasible path preferences} as 
$$\mc Z:=\{z \in \mathcal{S}_{\lambda}:\,y^z\in\mc F\}.$$ 
Moreover, let
the \emph{total latency} associated to a nonnegative vector $y$ in $\R_+^{\mc E}$ be 
\be\label{deef:latency} L(y)=\sum_{i\in\mc E}y_i\tau_i(y_i)\,.\ee  
Observe that the total latency $L(y)$ is finite if and only if the flow vector $y$ is feasible. In fact, as a consequence of Lemma \ref{lemma:convexity}, we have that the total latency function $L(y)$ is a strictly convex function of $y$ in $\mc F$.
Notice that, by the max-flow min-cut theorem (see \cite{Ahuja}, Thm. 4.1), the set of feasible flows $\mc F$ contains equilibrium $o$-$d$ flows if and only if the throughput $\lambda<C^{\text{min cut}}_{o,d}$, where 
$$C^{\text{min cut}}_{o,d}=\min_{\substack{\mc U\subseteq\mc V\,:\\o\in\mc U,\, d\notin\mc U}}\sum_{\substack{i\in\mc E\,:\\\theta_i\in\mc U,\,\kappa_i\notin\mc U}}C_i$$
is the min-cut capacity. It then follows that, for every $\lambda$ in $[0,C^{\text{min cut}}_{o,d})$, the total latency $L(y)$ admits a unique minimizer $y^*(\lambda)$ in the set of feasible equilibrium $o$-$d$ flows of throughput $\lambda$. We shall refer to such unique minimizer 
\be\label{SO-def}y^*(\lambda):=\argmin_{\substack{y\in\R_+^{\mc E}\\By=\lambda(\delta^{(o)}-\delta^{(d)})}}L(y)\ee as the \emph{social optimum} equilibrium flow.

\begin{example} Consider the  network in Figure~\ref{exemplegraph} with node set $\mathcal{V}=\{o,a,b,d\}$ and link set $\mathcal{E}=\{i_1, i_2, i_3, i_4, i_5, i_6\}$. It contains four distinct paths from $o$ to $d$. In fact, we may write $\Gamma = \{\gamma^{(1)}, \gamma^{(2)}, \gamma^{(3)},\gamma^{(4)}\}$, where $\gamma^{(1)} = (i_1,i_5)$, $\gamma^{(2)} = (i_2,i_6)$, $ \gamma^{(3)}=(i_1,i_3,i_6)$, and $\gamma^{(4)}=(i_2,i_4,i_5)$. Note that there is a cycle $\gamma^{(o)}=(i_3,i_4)$.  
	\begin{figure}
		\centering
		\begin{tikzpicture}
		[scale=1.1,auto=left,every node/.style={circle,draw=black!90,scale=.5,fill=white,minimum width=1cm}]
		\node (n1) at (0,0){\Large{o}};  
		\node (n2) at (2,1){\Large{a}}; 
		\node (n3) at (2,-1){\Large{b}}; 
		\node (n4) at (4,0){\Large{d}}; 
		\node [scale=0.8, auto=center,fill=none,draw=none] (n0) at (-0.8,0){};
		\node [scale=0.8, auto=center,fill=none,draw=none] (n5) at (4.8,0){};
		\foreach \from/\to in
		{n0/n1,n1/n2,n1/n3,n2/n4,n3/n4,n4/n5}
		\draw [-latex, right] (\from) to (\to); 
		\foreach \from/\to in
		{n2/n3,n3/n2}
		\draw [-latex, bend right] (\from) to (\to);
		\node [scale=2,fill=none,draw=none] (n5) at (1,0.7){${i_1}$};  
		\node [scale=2,fill=none,draw=none] (n6) at (1,-0.7){${i_2}$}; 
		\node [scale=2,fill=none,draw=none] (n7) at (1.45,0){${i_3}$}; 
		\node [scale=2,fill=none,draw=none] (n14) at (2.6,0){${i_4}$};
		\node [scale=2,fill=none,draw=none] (n8) at (3,0.7){${i_5}$};  
		\node [scale=2,fill=none,draw=none] (n9) at (3,-0.7){${i_6}$}; 
		\node [scale=1.5,fill=none,draw=none] (n10) at (-0.4,0.2){};
		\node [scale=1.5,fill=none,draw=none] (n15) at (4.4,0.2){};
		\node [scale=1,fill=none,draw=none] (n11) at (4.1,0){};
		\node [scale=1,fill=none,draw=none] (n12) at (2,1.1){};
		\node [scale=1,fill=none,draw=none] (n13) at (2,-1.1){}; 
		\end{tikzpicture}  
		\caption{\label{exemplegraph} Example of network with cycle.}  
	\end{figure}
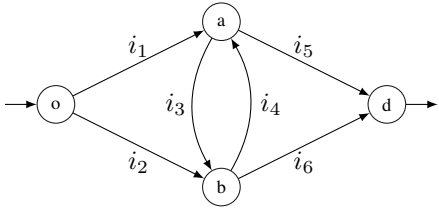
	For every link $i$ in $\mc E$, let the flow-density functions be given by 
	\begin{equation}\label{ex:phi}
		\varphi_i(x_i)=C_i(1-e^{-x_i})\,, \qquad x_i \in \R_+\,,
	\end{equation}
	where $C_i>0$ is link $i$'s capacity. 
	Then, the corresponding latency functions are given by
	\begin{equation}\label{ex:tau}
		\tau_i(y_i)=
		\begin{cases}
			\displaystyle{1}/{C_i} & \text{if} \ y_i=0\\
			\displaystyle\frac{1}{y_i}\log\left(\frac{C_i}{C_i-y_i}\right) & \text{if} \ 0<y_i<C_i \\
			+\infty & \text{if} \ y_i\geq C_i\,.
		\end{cases}
	\end{equation}
	Plots of the flow-density function \eqref{ex:phi} and of the latency function \eqref{ex:tau} are reported in Figure \ref{FigTraiett1}. 
	In the special case when the link capacities are 
	\be\label{ex:capacities}C_{i_1}=3\,,\ C_{i_2}=1\,,\ C_{i_3}=1\,,\ C_{i_4}=1\,,\ C_{i_5}=1\,,\ C_{i_6}=3\,,\ee 
	the min-cut capacity is $C^{\text{min cut}}_{o,d}=3$ and the minimum total latency and corresponding social optimum flow are plotted in Figure \ref{Mintotlatency} as a function of the throughput $\lambda$ in $[0,C^{\text{min cut}}_{o,d})$. 
	
	\begin{figure}
		\centering%
		\subfigure[\label{fig1Init}]%
		{\includegraphics[scale=0.30]{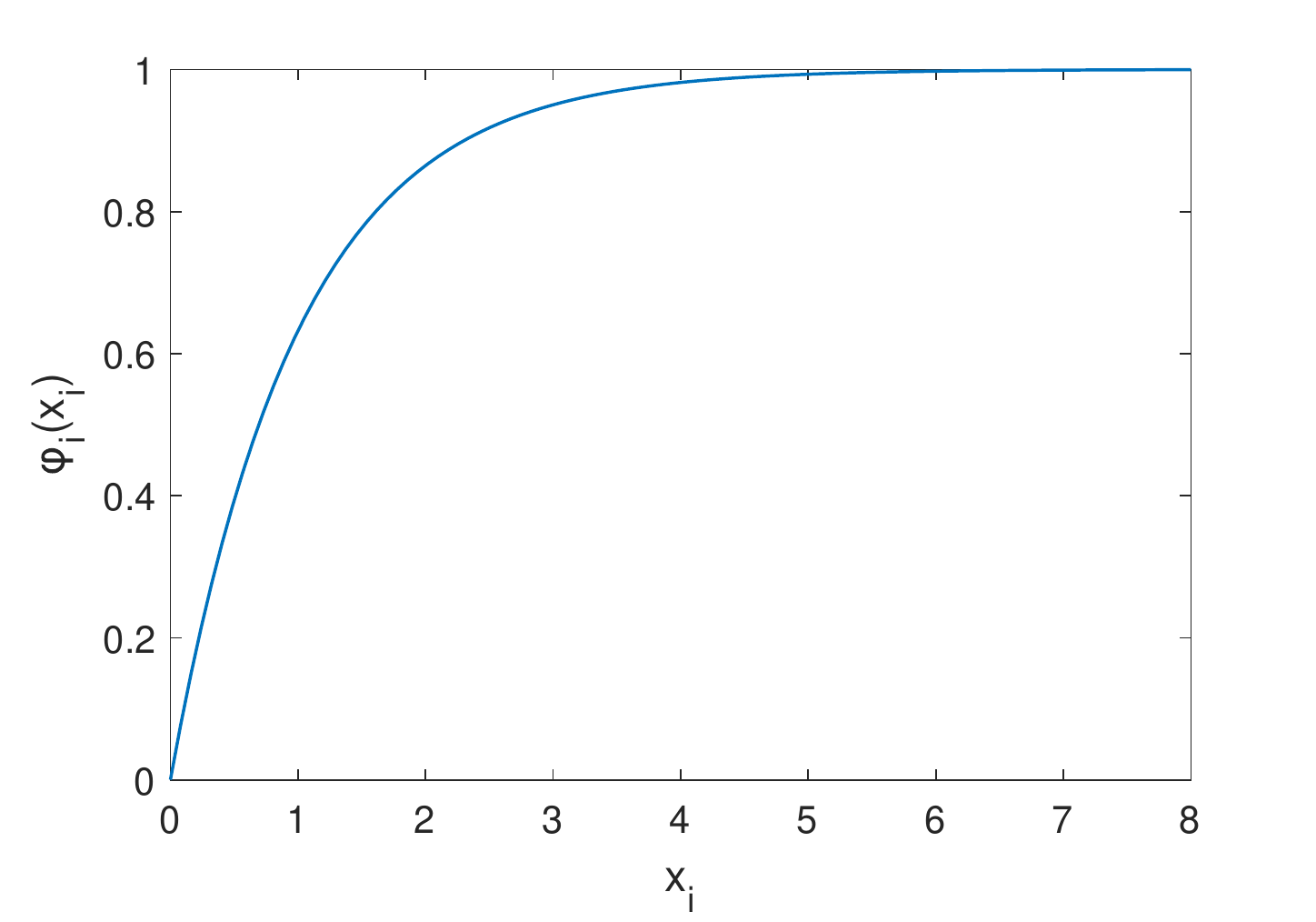}}
		\subfigure[\label{fig2Init}]%
		{\includegraphics[scale=0.29]{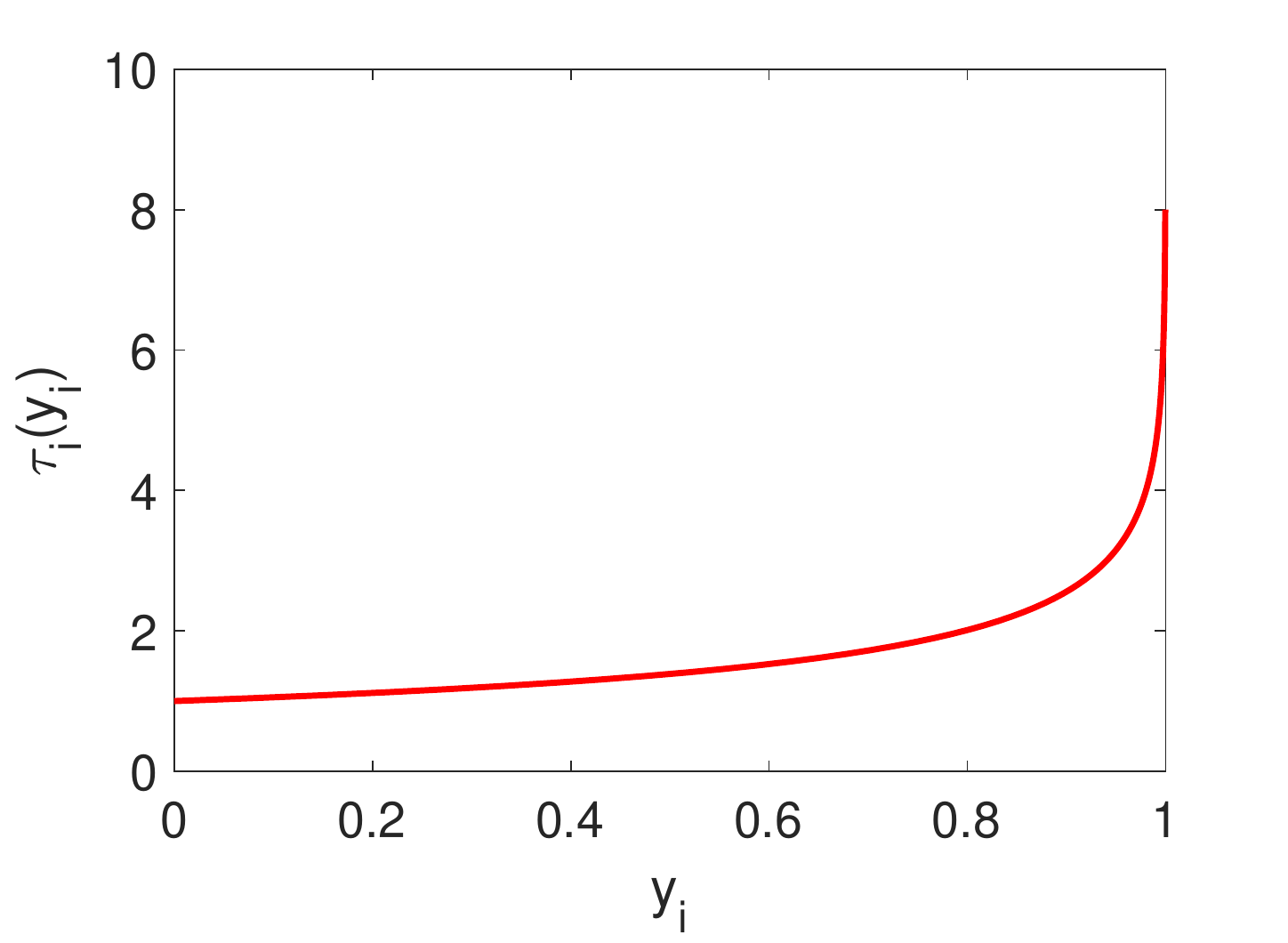}}
		\caption{\label{FigTraiett1} Plots of the flow-density function \eqref{ex:phi} in (a) and of the latency function \eqref{ex:tau} in (b), in the special case of capacity $C_i=1$.} 
	\end{figure}
\begin{figure}
	\centering%
	\subfigure[\label{fig1Init1}]%
	{\includegraphics[scale=0.29]{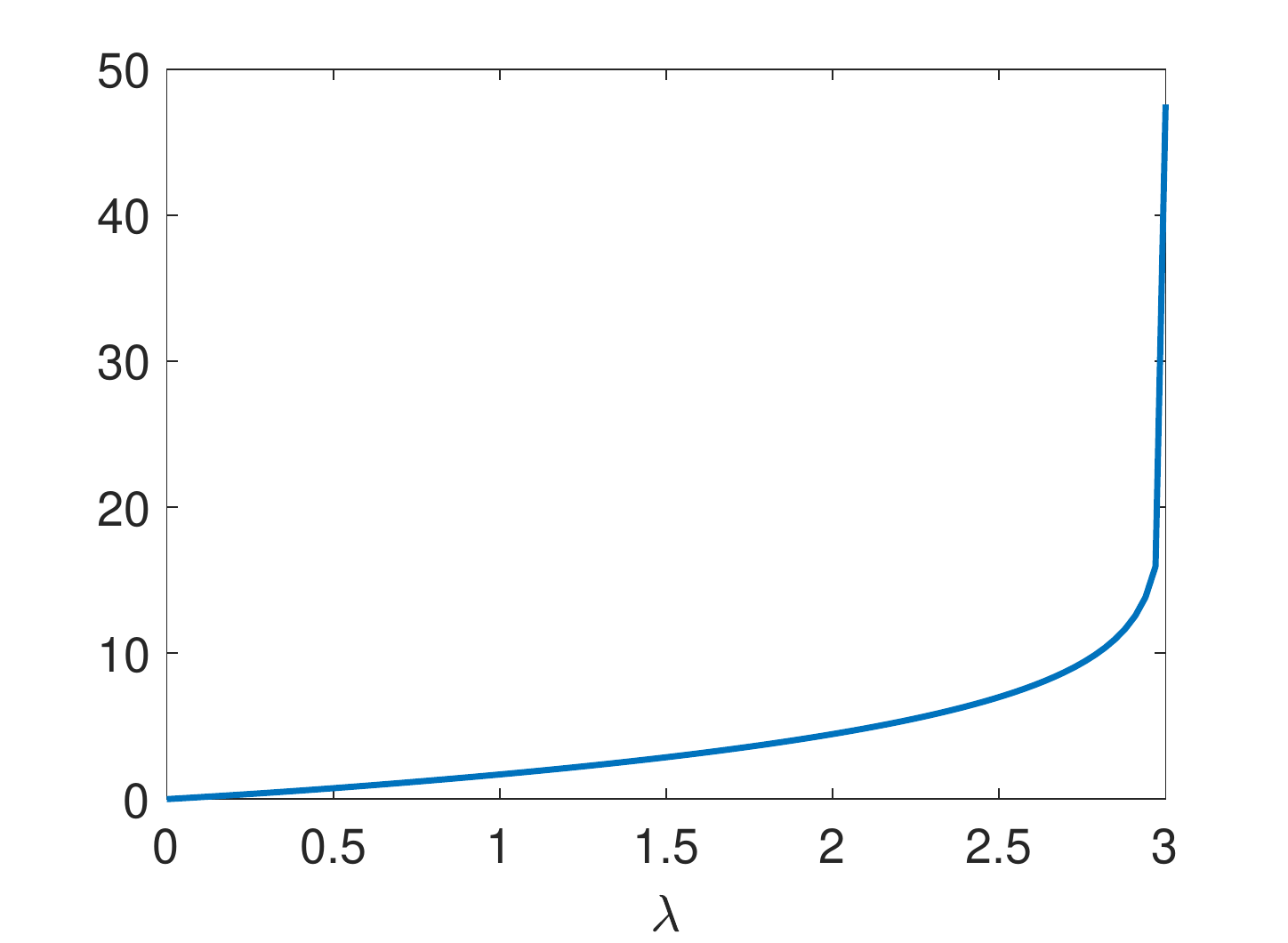}}
	\subfigure[\label{fig2Init2}]%
	{\includegraphics[scale=0.29]{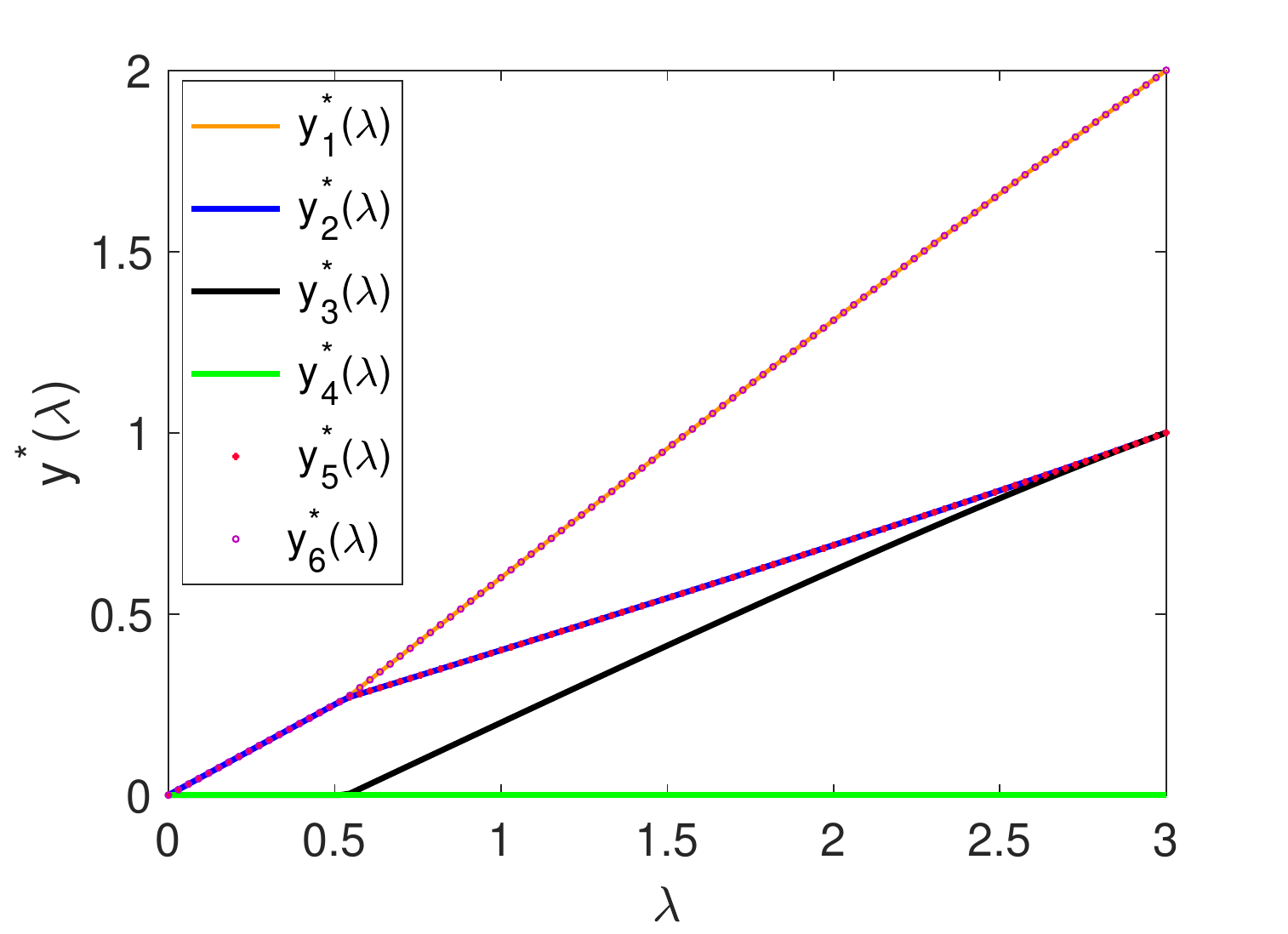}}
		\caption{\label{Mintotlatency} 
			In (a), plot of the minimum total latency as a function of the throughput $\lambda$ for a transportation network with topology as in Figure~\ref{exemplegraph}, flow-density functions as in \eqref{ex:phi}, and link capacities as in \ref{ex:capacities}. In (b), plots of the corresponding social optimum flow vector $y^*(\lambda)$. {In particular $y_6^*(\lambda)$ is overlapped to $y_1^*(\lambda)$, while $y_5^*(\lambda)$ is overlapped to $y_2^*(\lambda)$.}
} 
\end{figure}

\end{example}

\subsection{Multi-scale model of network traffic flow dynamics}

We shall consider a physical traffic flow  entering the network from the origin node $o$ at a constant rate $\lambda$, travelling on the different paths 
and finally exiting the network from the destination node $d$. 
Conservation of mass  implies that the density on every link $i$ in $\mc E$ at time $t\ge0$ evolves as
\be\label{mass-cons}\dot x_i(t)=\lambda\delta_{\theta_i}^{(o)}R_{oi}+ \sum_{j\in\mc E}R_{ji}(t)y_j(t)-y_i(t)\,,\ee
where \be\label{y(t)}y_i(t)=\varphi_i(x_i(t))\ee is the total outflow from link $i$, the terms $R_{ji}(t)$ and $R_{oi}(t)$ stand for the fractions of outflow from link $j$ and, respectively, from the origin node $o$, that moves directly towards link $j$, 
and the term $\lambda\delta_{\theta_i}^{(o)}$ accounts for the constant exogenous inflow in the origin node $o$. 
Topological constraints and mass conservation imply that $R_{ij}(t)=0$ whenever $\kappa_i\ne\theta_j$, i.e., whenever link $j$ is not immediately downstream of link $i$, that $R_{oj}(t)=0$ whenever $\theta_j\ne o$, and that 
$\sum_{j\in\mc E}R_{ij}(t)=1$ for $i=o$ and for every $i $ in $ \mc E$ such that $\theta_i\neq d$.
The matrix $R(t)=(R_{ij}(t))_{i,j\in\mc E}$ will be  referred to as the \emph{routing matrix}. 

Throughout, we shall assume that the routing matrix is determined by the path preferences that are continuously updated in response to available current traffic information and dynamic tolls. Formally, the relative appeal of the different paths is modelled by a time-varying nonnegative vector $z(t)$ in the simplex $\mc S_{\lambda}$,  to be referred to as the current \textit{aggregate path preference}.\footnote{Recall that $\mc S_{\lambda}$ stands for the simplex over the set of $o$-$d$-paths $\Gamma$, as defined in \eqref{def:Slambda}.} 
We shall assume that such aggregate path preferences determine the routing matrix as
\be\label{Rij=Gj}R_{ij}(t)=\left\{\ba{lcl}
G_j(z(t))&\se&\theta_j=\kappa_i\\
0&\se&\theta_j\ne\kappa_i\,,
\ea\right.\ee
for $ i,j$ in $\mc E$ and  $t\ge0$, where $G:\mc Z\to\R_+^{\mc E}$ is given by 
\begin{equation}\label{localchoice}
	G_{j}(z)=
	\begin{cases}
		\displaystyle\frac{y_j^{z}}{\displaystyle\sum_{i \in \mathcal{E}: \theta_i=\theta_j}y_i^z} & \text{if}\quad{\displaystyle\sum_{i \in \mathcal{E}: \theta_i=\theta_j}y_i^z}>0\\
		\,
		\displaystyle\frac{1}{\vert\{i \in \mathcal{E}: \theta_i=\theta_j\}\vert}&  \text{if}\quad {\displaystyle\sum_{i \in \mathcal{E}: \theta_i=\theta_j}y_i^z}=0\,,
	\end{cases}
\end{equation}
for each cell $j$ in $\mc E$.  Equations \eqref{Rij=Gj} and \eqref{localchoice} state that at every junction, represented by a node $v$ in $\mc V$, the outflow from every incoming cell $i$ such that $\kappa_i=v$ gets split among the cells $j$ immediately downstream (i.e., such that $\theta_j=v$) according to the proportion associated to the equilibrium flow vector $y^z$ corresponding to the path preference $z$, provided that $y^z$ is such there is flow passing through node $v$, and otherwise the split is uniform among the immediately downstream cells.  Notice that $G(z)$ as defined in \eqref{localchoice} is continuously differentiable on the interior of the set $\mc Z$, to be denoted as 
$$\mc Z^{\circ}:=\{z\in\mc Z:\,z_\gamma>0\,\forall\gamma\in\Gamma\}\,.$$

In the considered dynamical network traffic model, the aggregate path preference vector $z(t)$ is continuously updated as route decision makers access global information about the current traffic state of the whole network embodied by the vector 
\be\label{latencies-t}l(t)=(l_i(t))_{i\in\mc E}\,,\qquad l_i(t)=\tau_i(y_i(t))\,,\ee of current latencies on the different links.
The aggregate path preference vector is also influenced by a vector $w(t)=(w_i(t))_{i\in\mc E}$ of \emph{dynamic tolls}, that are to be determined by the transportation system operator. 
Specifically, let the cost perceived by each user, crossing a link $i$ in $\mathcal E$, be given by the sum of the latency $l_i(t)$ and the toll  $w_i(t)$ so that the perceived total cost that is expected to incur on a path $\gamma$ in $\Gamma$ assuming that the traffic levels on that path won't change during the journey is  $\sum_i A_{i\gamma}(l_i(t)+w_i(t))$. 
We shall then assume that the path preferences are updated at some rate $\eta>0$, according to a noisy best response dynamics 
\begin{equation}\label{evolpi}
	\dot{z}(t)=\eta\left(F^{(\beta)}(l(t),w(t))-z(t)\right)\,,
\end{equation}
where for every fixed uncertainty parameter $\beta>0$ the function $F^{(\beta)}:\R_+^{\mc E}\times\R_+^{\mc E}\to \mc Z$ is the perturbed best response defined as follows:

\begin{equation}\label{bestresponse}
	F^{(\beta)}(l,w)
	=\frac{\lambda\exp(-\beta(A'(l+w)))}{\mathbf{1}'\exp(-\beta(A'(l+w)))}.
\end{equation}

We shall rewrite the coupled dynamics of the physical flow and the path preferences defined in \eqref{mass-cons}--\eqref{bestresponse} in the compact notation
\begin{equation}\label{sistemaccoppiato}
	\left\{\ba{lcl}
	\dot{x}(t)&=&H(y(t), z(t))\,,\qquad y(t)=\varphi(x(t))\,,\\[7pt]
	\dot{z}(t)&=&\eta\left(F^{(\beta)}\left(l(t),w(t)\right)-z(t)\right)\,,
	\ea\right.
\end{equation}
where $H:\mc F\times\mc Z\to\R^{\mc E}$ is defined as 
\begin{equation}\label{H}
	H_i(y, z):= G_{i}(z)\bigg(\lambda\delta_{\theta_i}^{(o)}+\sum_{j: \kappa_j=\theta_i}y_j\bigg)-y_i\,,\qquad i\in\mc E\,.
\end{equation}

\section{Problem statement and main results}\label{section3}
The goal of this paper is to design robust scalable feedback pricing policies 
\be\label{omega}\omega:\mc F\to\R_+^{\mc E}\ee
determining in real time the dynamic tolls 
\be\label{w=omega}w(t)=\omega(y(t))\ee
with the objective of guaranteeing stability and achieving social optimality 
for the closed-loop network traffic flow dynamics \eqref{sistemaccoppiato}---\eqref{w=omega}. 

Observe that, for any given fixed inflow vector $\lambda\delta^{(o)}$ and constant toll vector $w$, and in the special case of cycle-free network topology, stability and convergence to the corresponding Wardrop equilibrium ---as defined later in this section--- follow from the results in \cite{ComoSavla}. In fact, given full knowledge of the exogenous inflow $\lambda\delta^{(o)}$ and of the whole transportation network characteristics, one could use classical results in order to pre-compute static tolls that would align such Wardrop equilibrium with the social optimum. However, even for cycle-free networks, such an approach would result in an inadequate solution as it would lack robustness with respect to the value of the  exogenous input $\lambda\delta^{(o)}$, as well as to changes of the network characteristics in response, e.g., to accidents and other disruptions. 

In contrast, we seek to design feedback pricing policies that are universal with respect to values of the exogenous inflow and robustly adapt in real time to changes of the  network characteristics. We shall particularly focus on the class of \emph{decentralized monotone feedback pricing policies}, as defined below. 
\begin{definition}In a transportation network with topology $\mc G=(\mc V,\mc E)$, a feedback pricing policy $\omega:\mc F\to\R_+^{\mc E}$ is said to be:
\begin{enumerate}
\item[(i)]
 \emph{monotone} if $\omega(y)\ge\omega(y')$ for every $y,y'$ in $\mc F$ such that $y\ge y'$, where inequalities are meant to hold true entrywise; 
\item[(ii)]
\emph{decentralized} if, for every $i$ in $\mc E$, the toll $w_i=\omega_i(y)$ is a function of the flow $y_i$ on link $i$ only. 
\end{enumerate}
\end{definition}

Throughout the rest of the paper, we shall emphasize the local structure of decentralized pricing policies by writing $w_i=\omega_i(y_i)$, with a slight abuse of notation. 
As shown in the following, such robust fully local feedback pricing policies can be designed with global guarantees on stability and optimality. 
Before stating our main results, we introduce the notion of generalized Wardrop equilibrium with feedback pricing. 
\begin{definition}(Generalized Wardrop equilibrium with feedback pricing). \label{WED}
For a transportation network with topology $\mc G=(\mc V,\mc E)$ and latency functions $\tau_i$,  let $o$ and $d$ in $\mc V$, with $d\ne o$ reachable from $o$, be an origin and a destination, respectively. Let  $\Gamma$ the set of $o$-$d$ paths and $A$ the link-path incidence matrix. 
 Then, for a feedback pricing policy $\omega:\mc F\to\R_+^{\mc E}$, an $o$-$d$ equilibrium flow vector $y$ in $\mc F$ of throughput $\lambda$ is a \emph{generalized Wardrop equilibrium} if $y=Az$ for some path preference vector $z $ in $ \mc S_{\lambda}$ such that for every path $\gamma $ in $ \Gamma$ with $z_{\gamma}>0$, we have  
	\begin{equation}\label{Wardrop}
		\left(A'\left(\tau(y)+\omega(y)\right)\right)_{\gamma} \leq \left(A'\left(\tau(y)+\omega(y)\right)\right)_{\tilde\gamma} \quad \forall \tilde\gamma \in \Gamma. 
	\end{equation}
\end{definition}

Equation \eqref{Wardrop} states that the sum of the total delay and the total toll associated to an $o$-$d$ path $\gamma$ at the equilibrium flow $y$ are less than or equal to the sum of the total delay and the total toll associated to any other $o$-$d$ path $\tilde\gamma$. Hence, a generalized Wardrop equilibrium with feedback pricing is characterized as being the flow associated to a path preference vector supported on the subset of paths with minimal sum of total latency plus total toll. In the special case with no tolls, i.e., when the feedback pricing policy $\omega(y)\equiv0$, this reduces to the classical notion of Wardrop equilibrium \cite{Wardrop}. More in general, for constant tolls $\omega(y)\equiv w$ we get the standard notion of Wardrop equilibrium with tolls. For general decentralized monotone feedback pricing policies, existence and uniqueness of a generalized Wardrop equilibrium are guaranteed by the following result, proven in Appendix \ref{proof:prop-Wardrop}.

\begin{proposition}\label{prop:Wardrop}
Consider a transportation network with topology $\mc G=(\mc V,\mc E)$ and strictly increasing latency functions.
Let $o$ and $d$ in $\mc V$, with $d\ne o$ reachable from $o$, be an origin and a destination, respectively. 
Then, for every throughput $\lambda$ in $[0,C^{\text{min cut}}_{o,d})$ and every decentralized monotone feedback pricing policy $\omega:\mc F\to\R_+^{\mc E}$, there exists a unique generalized Wardrop equilibrium $y^{(\omega)}$ and it can be characterized as the solution of the convex optimization problem 
\begin{equation}\label{wardopcomeminimo}
y^{(\omega)}=\operatornamewithlimits{\arg\min}_{\substack{y\in\R_+^{\mc E} \\ By=\lambda(\delta^{(o)}-\delta^{(d)})}}\sum_{i \in \mathcal{E}}D_i(y_i)\,,
\end{equation}  
 where, for each link $i$ in $\mc E$,
\be\label{Didef}D_i(y_i)=\int_0^{y_i}\left(\tau_i(s)+\omega_i(s)\right)\de s\ee 
is the primitive of the perceived cost $\tau_i(y_i)+\omega_i(y_i)$. 
\end{proposition}

\begin{remark} It is worth noticing that it is possible to modify the definition of perceived cost by weighing $\tau_i$ differently from $\omega_i$. This modification would cause no restriction on the validity of our results. 
\end{remark}

In the following, 
we shall prove that for small $\eta$ and large $\beta$, the long-time behavior of the system \eqref{sistemaccoppiato} is approximately at Wardrop equilibrium which, under proper distributed feedback pricing policies, coincides with the social optimum equilibrium.
The following is the main result of the paper. It will be proved in the next section using a singular perturbation approach.
\begin{theorem}\label{theorem:main}
	{Consider a transportation network with topology $\mc G=(\mc V,\mc E)$ and flow-density functions satisfying Assumption \ref{assumption:flow-function}. Let  $\lambda$ in $[0,C^{\text{min cut}}_{o,d})$ be the throughput and $\omega:\mc F\to\R_+^{\mc E}$ be a Lipschitz-continuous monotone decentralized feedback pricing policy.} Then, 
	there exists a perturbed equilibrium flow $y^{(\omega,\beta)} $ in $ \mathcal{F}$ such that, for every initial condition $(z(0), x(0)) $ in $ \mc Z^{\circ}\times \R_+^\mathcal{E}$, the solution of the closed-loop network traffic flow dynamics \eqref{sistemaccoppiato}---\eqref{w=omega} satisfies 
	\begin{equation}
	\limsup_{t\to \infty}\ \lVert y(t)-y^{(\omega,\beta)}\rVert\leq \bar\delta(\eta)\,,\qquad \eta>0\,,
	\end{equation}
	where $\bar\delta(\eta)$ is a nonnegative-real-valued, nondecreasing function such that $\lim_{\eta\to 0}\bar\delta(\eta)=0$. 
	Moreover, 
	\begin{equation}\label{limiteperturbazioni}
	\lim_{\beta\to \infty}y^{(\omega,\beta)}= y^{(\omega)}.
	\end{equation}
\end{theorem}
\vspace{0.3cm}
Theorem \ref{theorem:main} states that the system planner globally stabilizes the transportation network around the Wardrop equilibrium using non-decreasing decentralised \CorrR{flow-dependent tolls. Notice that the case  $\lambda\geq C^{\text{min cut}}_{o,d}$ is not covered by Theorem \ref{theorem:main} and in fact in that case one can show that the transportation system would become unstable as time grows large (see e.g., \cite{{Robust1}}).}
\begin{remark}
	{
	Notice that, even in the cycle-free case, Theorem \ref{theorem:main} does not follow from Theorem 2.5 in \cite{ComoSavla} if the tolls are not constant. Indeed, although the functions $\tau$ and $\omega$ both depend on the flow $y$, it is not always possible consider an
	auxiliary function $\bar \tau=\tau+\omega$ and directly apply
	the result from \cite{ComoSavla} due to the specific structure imposed on $\tau$ in \eqref{delayfunction}. The feedback structure of the considered closed-loop multiscale transportation network dynamics is illustrated in Figure~\ref{blocco}. }
\end{remark}
\begin{figure}
	\centering
	\includegraphics[scale=0.24]{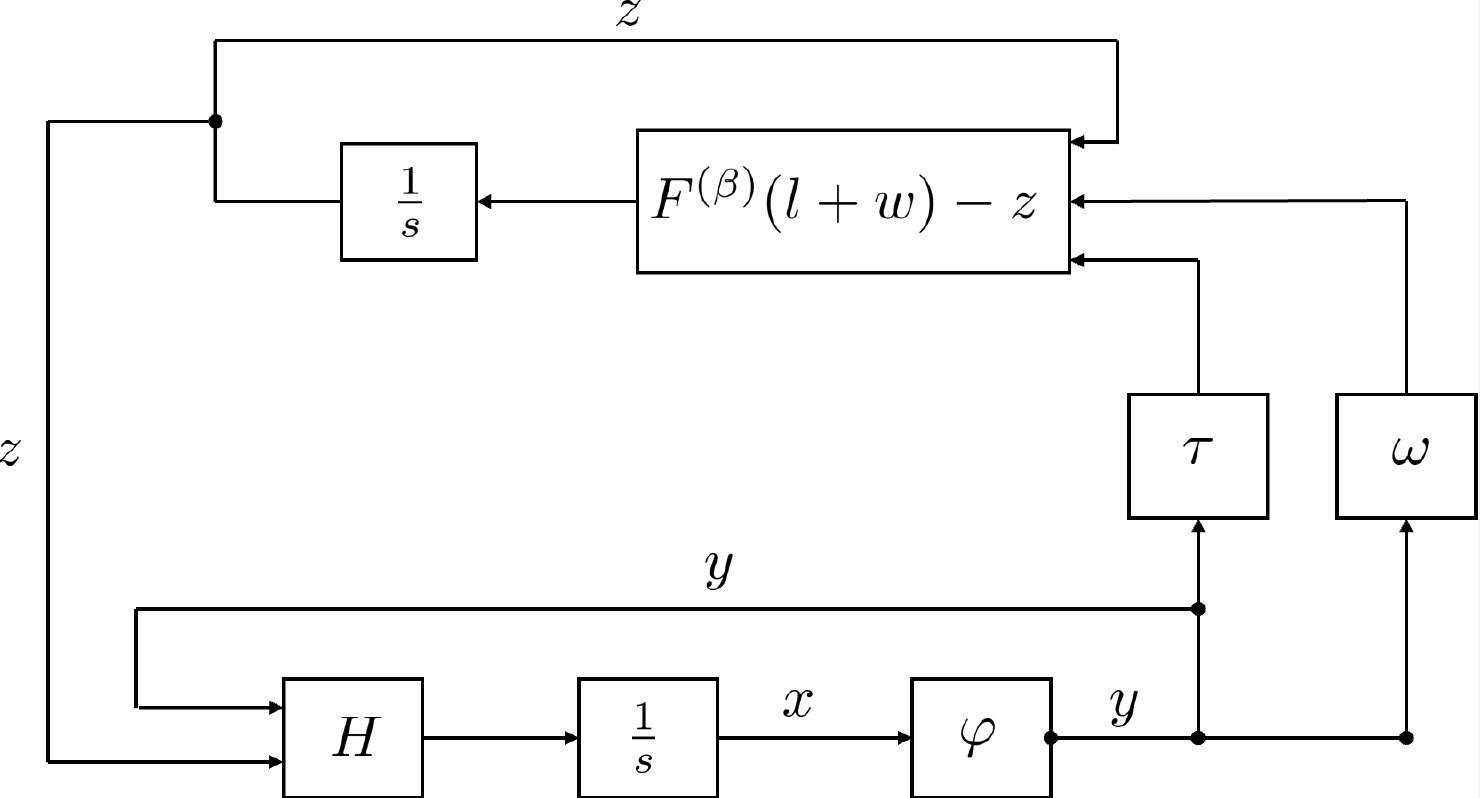}
	\caption{\label{blocco} Block diagram of the problem.}
\end{figure}

Now, we focus on the special case of decentralized feedback tolls the marginal cost tolls, namely, when 
\begin{equation}\label{marginalcost}
w_i(t)=\omega_i(y_i(t))=y_i(t)\tau'_i(y_i(t))\,, \qquad i \in  \mathcal{E}.
\end{equation}
Due the properties of the delay function $\tau_i$, the marginal cost tolls $\omega_i(y_i(t))$ defined in \eqref{marginalcost} are increasing functions of the flow $y_i(t)$, so that Theorem \ref{theorem:main} applies in this case. Moreover, the following additional result holds true. 
\begin{corollary}\label{Corollary}
Consider a transportation network with topology $\mc G=(\mc V,\mc E)$ and flow-density functions satisfying Assumption \ref{assumption:flow-function}. Let  $\lambda$ in $[0,C^{\text{min cut}}_{o,d})$ be the throughput and $\omega:\mc F\to\R_+^{\mc E}$ be the dynamic feedback marginal cost tolls defined in \eqref{marginalcost}. Then, the transportation network globally stabilizes around the social optimum traffic assignment $y^*(\lambda)$, i.e., for every initial condition $(z(0), x(0)) $ in $ \mc Z^{\circ}\times \R_+^\mathcal{E}$, the solution of the closed-loop network traffic flow dynamics \eqref{sistemaccoppiato}---\eqref{w=omega} satisfies 
\begin{equation}\label{limitealsocial}
\lim_{\beta\to \infty}y^{(\omega,\beta)}= y^*(\lambda)\,.
\end{equation}
\end{corollary}
\begin{proof}
First notice that with feedback marginal cost tolls $\omega_i(y_i)=y_i\tau'_i(y_i)$, 
the perceived cost $\tau_i(y_i)+\omega_i(y_i)$ on each link $i$ in $\mc E$ has primitive 
$$D_i(y_i)=\int_0^{y_i}\bigg(\tau_i(s)+s\tau'_i(s)\bigg)\de s=y_i\tau_i(y_i)\,,$$
so that 
$$\sum_{i\in\mc E}D_i(y_i)=L(y)$$
coincides with the total latency. 
It then follows from the characterization \eqref{wardopcomeminimo} of Proposition \ref{prop:Wardrop}
that 
$$
y^{(\omega)}=\!\!\!\!\!\!
\operatornamewithlimits{\arg\min}_{\substack{y\in\R_+^{\mc E} \\ By=\lambda(\delta^{(o)}-\delta^{(d)})}}
\!\!\sum_{i \in \mathcal{E}}D_i(y_i)=\!\!\!\!\!\!
\operatornamewithlimits{\arg\min}_{\substack{y\in\R_+^{\mc E} \\ By=\lambda(\delta^{(o)}-\delta^{(d)})}}\!\!L(y)=
y^*(\lambda)
\,.
$$
The claim then follows as a direct application of Theorem \ref{theorem:main}. 
\end{proof}
\begin{remark}\label{rem}
 It can be easily seen that Corollary \ref{Corollary} holds true also if the dynamic feedback marginal cost tolls \eqref{marginalcost} are replaced by 
 the constant marginal cost tolls
	 \begin{equation}\label{constolls}
	 w_i^{*}=y_i^{*}(\lambda)\tau'_i(y_i^{*}(\lambda))\,, \qquad i \in \mathcal{E}\,.
	 \end{equation} 
However, in contrast to the dynamic feedback marginal cost tolls \eqref{marginalcost},
such constant marginal cost tolls \eqref{constolls} require knowledge both of the social optimum flow and the exogenous inflow $\lambda\delta^{(o)}$ and lack robustness with respect to changes of the value of $\lambda$, as well as to changes of the network. 

\end{remark}
\begin{remark}
	In order to implement the dynamic feedback marginal cost tolls \eqref{marginalcost}, each local controller is required to compute the product $y_i\tau_i'(y_i)$ of the link's current flow times the link's latency function's derivative. 
Notice that, using \eqref{delayfunction}, we get 
	\begin{equation*}
	\omega_i(y_i)=\frac{1}{\varphi'_i(x_i)}-\frac{x_i}{y_i}\,, 
	\end{equation*}
	Hence, it is possible to reduce the computation of $\tau_i'$ to that of the derivative of the flow-density function $\varphi_i$.
	
\end{remark}
\section{Proof of Theorem \ref{theorem:main}}\label{section4}
In this section, we prove Theorem \ref{theorem:main}. First of all, notice that since the functions $F^{(\beta)},$ $G$, and $\varphi$ are differentiable, standard results imply the existence and uniqueness of a solution of the initial value problem associated to \eqref{sistemaccoppiato},  
with initial condition $(z(0), x(0)) $ in $ \mc Z^\circ\times\R_+^\mathcal{E}$. 
In order to prove the stability result, we shall adopt a singular perturbation approach. Our strategy consists in thinking of the path preference vector $z$ as quasi-static when we analyse the fast-scale dynamics \eqref{mass-cons}, and considering the flow vector $y$ almost equilibrated (i.e., close to $y^{z}$) when study the slow-scale dynamics \eqref{evolpi}. Below, we will derive a series of intermediate results that will then be combined to prove Theorem \ref{theorem:main}.
 
 Before proceeding, we introduce some notation to be used throughout the section. Similar to \eqref{latencies-t} and \eqref{w=omega} let
\begin{equation*}
l^z(t)=(l^z_i(t))_{i\in\mc E}\,,\qquad l^z_i(t)=\tau_i(y^z_i(t))
\end{equation*}
and 
\begin{equation*}
w^z(t)=(w^z_i(t))_{i\in\mc E}\,,\qquad w^z_i(t)=\omega_i(y^z_i(t))
\end{equation*}
be respectively the vector of current latencies and the one of dynamic tolls both corresponding to the flow $y^z$ associated to the path preference $z$. 

Furthermore, observe that the perturbed best response function \eqref{bestresponse} satisfies 
	\begin{equation}\label{gen-bestresponse}
	F^{(\beta)}(l,w):=\operatornamewithlimits{\arg\min}_{\alpha \in \mc Z_h}\{\alpha'A'(l+w)+h(\alpha)\},
	\end{equation}
	where $h:\mc Z\to \mathbb{R}$ is the negative entropy function defined as 
\begin{equation}\label{negativentropy}
h(z):=\beta^{-1}\sum_{\gamma \in \Gamma} z_{\gamma} \log z_{\gamma}\,,
\end{equation}
using the standard convention that $0\log0=0$. In fact, all our analysis and results apply to a more general setting where the perturbed best response function is defined as 
	\begin{equation}\label{generalbestresponse}
	F^{(h)}(l,w):=\operatornamewithlimits{\arg\min}_{\alpha \in \mc Z_h}\{\alpha'A'(l+w)+h(\alpha)\},
	\end{equation}
 for some \emph{admissible perturbation}
	$h:\mc Z_h\to \mathbb{R}$ such that $\mc Z_h\subseteq \mc Z$ is a closed convex set, $h(\cdot)$ is strictly convex, twice differentiable in the interior $\mc Z_h^{\circ}$ of $\mc Z_h$, and $\lim_{z\to \partial \mc Z_h}\lVert\tilde{\nabla}h(z)\rVert=\infty$. These conditions on $h$ imply that $F^h(l,w)$ belongs to $\mc Z_h^{\circ}$ and that it is continuously differentiable on $\R_+^{\mc E}\times\R_+^{\mc E}$.  Notice that clearly the negative entropy function \eqref{negativentropy} is an admissible perturbation as defined above. We shall then proceed to proving Theorem \ref{theorem:main} in this more general setting. 

 Now, let 
\begin{equation*}
x_i^{z}:=\varphi_i^{-1}(y_i^{z}), \qquad \sigma_i:=\text{sgn}(x_i-x_i^{z})=\text{sgn}(y_i-y_i^{z})
\end{equation*}
denote, respectively, the density corresponding to the flow associated to the path preference $z$ and the sign of the difference between it and the actual density $x_i$. Then, we define the functions
\begin{equation}\label{Lyapfunz}
V(y, z)=\lVert y-y^{z}\rVert_1, \quad \text{and} \quad  W(x, z)=\lVert x-x^{z}\rVert_1.
\end{equation}
The following technical results aim at showing that \eqref{Lyapfunz} are Lyapunov functions for the fast-scale dynamics \eqref{mass-cons} with stationary path preference $z$.
\begin{lemma}\label{lemma:l1}\label{valorecostante} Let $\overline{\mathcal{E}}\subseteq\mc E$ be a nonempty set of cells. Then, 
\be\label{maxv}
	\max_{j \in \overline{\mathcal{E}}}\Big\{1-\sum_{\substack{i\in{\overline{\mathcal{E}}:} \\ \theta_i=\kappa_j}}G_i(z)\Big\}
\ge\frac1{|\mc V|}
\ee

\end{lemma}
\begin{proof}
Let $\overline{\mathcal{V}}=\{v \in\mc V: v=\kappa_i, i\in \overline{\mathcal{E}}\} $. 
Observe that 
$$\sum_{\substack{i\in\ov{\mc E}\\ \theta_i=d}}G_i(z)=0\,,$$
so that, if $d$ in $ \overline{\mathcal{V}}$ then 
$$\max_{j \in \overline{\mathcal{E}}}\Big\{1-\sum_{\substack{i\in{\overline{\mathcal{E}}}: \\ \theta_i=\kappa_j}}G_i(z)\Big\}=1\,,$$ 
and the claim follows immediately.

We can then focus on the case when $d \notin \overline{\mathcal{V}}$. 
Let 
\begin{equation}\label{1alpha}
\alpha=\sum_{\substack{i:\kappa_i\in\overline{\mathcal{V}} \\ \theta_i \notin  \overline{\mathcal{V}}}}y_i^z+\lambda\delta_i^{(o)}
\end{equation}
be the total inflow in $\overline{\mathcal{V}}$ which is also equal to the total outflow from $\overline{\mathcal{V}}$. Indeed $\alpha$ in \eqref{1alpha} can be also written as
\begin{equation}\label{2step}
\alpha=\sum_{\substack{i:\kappa_i\notin\overline{\mathcal{V}} \\ \theta_i \in  \overline{\mathcal{V}}}}y_i^z= \sum_{v \in \overline{\mathcal{V}}}\sum_{\substack{i:\kappa_i\notin\overline{\mathcal{V}} \\ \theta_i=v}}y_i^z \leq \sum_{v \in \overline{\mathcal{V}}}\sum_{\substack{i\notin{\overline{\mathcal{E}}} \\ \theta_i=v}}y_i^z
\end{equation}
Now, let $$\displaystyle f_v=\sum_{i:\kappa_i=v}y_i^z$$ be outflow from a single node $v$ and observe that $f_v\le\alpha$ for every node $v$. 
Using this and  \eqref{2step} we get
\begin{equation}
\alpha\leq \sum_{v \in \overline{\mathcal{V}}}\sum_{\substack{i\notin{\overline{\mathcal{E}}} \\ \theta_i=v}}y_i^z=\sum_{v \in \overline{\mathcal{V}}}f_v\sum_{\substack{i\notin{\overline{\mathcal{E}}} \\ \theta_i=v}}G_i(z)\leq \alpha\sum_{v \in \overline{\mathcal{V}}}\sum_{\substack{i\notin{\overline{\mathcal{E}}} \\ \theta_i=v}}G_i(z).
\end{equation}
Hence,
\begin{equation}
\frac1{|\mc V|}\le \frac1{|\ov{\mc V}|}\le \frac1{|\ov{\mc V}|}\sum_{v \in \overline{\mathcal{V}}}\sum_{\substack{i\notin{\overline{\mathcal{E}}} \\ \theta_i=v}}G_i(z)
\le\max_{v \in \overline{\mathcal{V}}}\sum_{\substack{i\notin{\overline{\mathcal{E}}} \\ \theta_i=v}}G_i(z)\,,
\end{equation}
so that 
$$
\max_{j \in \overline{\mathcal{E}}}\Big(1-\sum_{\substack{i\in{\overline{\mathcal{E}}} \\ \theta_i=\kappa_j}}G_i(z)\Big) 
=\max_{v \in \overline{\mathcal{V}}}\sum_{\substack{i\notin{\overline{\mathcal{E}}} \\ \theta_i=v}}G_i(z)
\ge\frac1{|\mc V|}\,,
$$
hence proving the claim. 
\end{proof}
\begin{lemma}\label{gradienteW}
	For every $y=\varphi(x)$ in $ \mathcal{F}$ and $z $ in $ \mc Z$
	\begin{equation*}
	\nabla_{x}W(x, z)'H(y, z) \leq -\varsigma V(y,z),
	\end{equation*}
		where $\varsigma={1}/{\vert\mathcal{V} \vert\vert \mathcal{E}\vert}$.
\end{lemma}
\begin{proof}
	Observe that by \eqref{localchoice} we get $$y_i^z= G_{i}(z)\bigg(\lambda\delta_{\theta_i}^{(o)}+\sum_{j: \kappa_j=\theta_i}y_j^z\bigg).$$ 
	We will use the above in the second equality of the computation below. Indeed we have	
\begin{equation}\label{differenzaV}
\begin{aligned}
&\nabla_{x}W(x, z)'H(y, z) = \\
&	\sum_{i \in \mathcal{E}}\sigma_i\Big( G_{i}(z)\bigg(\lambda\delta_{\theta_i}^{(o)}+\sum_{j: \kappa_j=\theta_i}y_j\bigg)-y_i\Big) \\
	&= \sum_{i \in \mathcal{E}}\sigma_i\bigg( G_{i}(z)\bigg(\lambda\delta_{\theta_i}^{(o)}+\sum_{j: \kappa_j=\theta_i}y_j\bigg)\\ & - G_{i}(z)\bigg(\lambda\delta_{\theta_i}^{(o)}+\sum_{j: \kappa_j=\theta_i}y_j^z\bigg)\bigg)+ \sum_{i \in \mathcal{E}}\sigma_i(y_i^z-y_i)\\
	&=\sum_{i \in \mathcal{E}}\sigma_i\Big(G_{i}(z)\sum_{j: \kappa_j=\theta_i}(y_j-y_j^z)\Big)-\sum_{i \in \mathcal{E}}\sigma_i(y_i-y_i^z).
	\end{aligned}
	\end{equation}
Now, define 
$$\overline{\mathcal{E}}=\{i \in \mathcal{E}: \sigma_i\neq 0 \}$$
and put $$\delta_i=\vert y_i-y_i^z\vert\,,\qquad i\in\mc E\,.$$ 
We have that
	\begin{equation*}\label{stimastellata}
	\displaystyle \delta_i\geq \min_{k \in \overline{\mathcal{E}}} \delta_k\geq\frac{\Vert \delta\Vert_1}{\vert \mathcal{E}\vert}, \quad \forall \ i \in \overline{\mathcal{E}}.
	\end{equation*}
	Then by \eqref{differenzaV}
	\begin{equation}\label{minorazione}
	\begin{aligned}
	&\sum_{i \in \mathcal{E}}\sigma_i\Big(G_{i}(z)\sum_{j: \kappa_j=\theta_i}(y_j-y_j^z)\Big)-\sum_{i \in \mathcal{E}}\sigma_i(y_i-y_i^z) \\
	&\leq \sum_{i \in\overline{\mathcal{E}}}\Big(G_i(z)\sum_{j\in\overline{\mathcal{E}}: \kappa_j=\theta_i}\delta_j\Big)-\sum_{i \in\overline{\mathcal{E}}}\delta_i  \\
	&= -\sum_{j \in\overline{\mathcal{E}}} \delta_j\Big(1-\sum_{i\in\overline{\mathcal{E}}: \theta_i=\kappa_j} G_i(z)\Big)\\
	&\leq -\frac{\Vert \delta\Vert_1}{\vert \mathcal{E}\vert} \max_{j \in\overline{\mathcal{E}}}\Big( 1-\sum_{i\in\overline{\mathcal{E}}: \theta_i=\kappa_j} G_i(z)\Big)\\
    &\le -\frac{||\delta||_1}{|\mc V||\mc E|}=-{\varsigma V(y,z)}
	\end{aligned}
	\end{equation}
	by using Lemma \ref{valorecostante}
\end{proof}

The following two results show that both $y_i^z(t)$ and $y_i(t)$ stay bounded away from the maximum flow capacity $C_i$.
\begin{lemma}\label{limitatezzafpi}
	Given the admissible perturbation \eqref{negativentropy}, there exists $t_0 $ in $ \mathbb{R}_+ $ and, for every link $i $ in $ \mathcal{E}$, a finite positive constant $\overline{C}_i$, dependent on $h$, but not on $\eta$, such that for every initial condition $(z(0), x(0)) $ in $ \mc Z^{\circ}\times \R_+^\mathcal{E}$, 
	$$
	y_i^{z}(t)\leq \overline{C}_i < C_i \qquad \forall t\geq t_0,\  \forall i \in \mathcal{E}. 
	$$
\end{lemma}
\begin{proof}
	The fact that  $y_i^z(t)\leq \lambda$ for all $i $ in $ \mathcal{E}$ follows from the fact that the arrival rate at the origin is unitary. Hence, for all $i $ in $ \mathcal{E}$ with $C_i>\lambda$ (and therefore also for $C_i=\infty$) the claim follow with $\overline{C}_i=\lambda$ and $t_0=0$. We now consider the case when $C_i < \lambda$ for all $i $ in $ \mathcal{E}$. Recall that by the definition of admissible perturbation, the domain of \eqref{negativentropy} is a closed set $\mc Z_\beta \subseteq \mc Z^{\circ}$. This implies that
	$$\xi_i:=C_i-\sup\{(A\alpha)_i: \alpha \in \mc Z_\beta\}>0 .$$
	It follows from \eqref{bestresponse} that
	\begin{equation*}
		C_i-\xi_i = \sup\{(A\alpha)_i: \alpha \in \mc Z_\beta\}\geq \sup\{(AF^{(\beta)}(l,w))_i\}.
	\end{equation*}
	Hence, one gets
	\begin{equation*}
	\frac{d}{dt}y_i^z(t)=\eta(A(F^{(\beta)}(l(t),w(t))-z(t)))_i\leq \eta(C_i-\xi_i-y_i^z).
	\end{equation*}
	This implies that 
	\begin{equation}\label{minorazionifz}
	y_i^z(t)-C_i+\xi_i\leq (y_i^z(0)-C_i+\xi_i)e^{-\eta t}\leq \lambda e^{-\eta t}, \ t\geq 0,
	\end{equation}
	where the last inequality comes from the fact that $y_i^z(0)\leq \lambda$ and $C_i\geq \xi_i$. For $i$ in $ \mathcal{E}$ with $C_i< \lambda$  the claim now follows from \eqref{minorazionifz} by choosing, for example, $\overline{C}_i:= C_i-\xi/2$ with $\xi:=\min\{ \xi_i: i \in \mathcal{E}\ \text{s.t.} \ C_i< \lambda\}$ and $t_0:=-\eta^{-1}\log(\xi/2\lambda)$.
\end{proof}
\begin{lemma}\label{limitatezzay}
Given the admissible perturbation \eqref{negativentropy}, there exist some $\eta^*>0$ and $\tilde{C}_i>0$ for $i $ in $ \mathcal{E}$, such that for every $\eta<\eta^*$ 
and every initial condition $(z(0), x(0)) $ in $ \mc Z^\circ\times \R_+^\mathcal{E}$, 
$$
y_i(t)\leq \tilde{C}_i < C_i \qquad \forall t\geq 0,\ \forall i \in \mathcal{E}. 
$$
\end{lemma}
\begin{proof}
For $t\geq 0$, let us define 
$$\zeta(t):=W(x(t), z(t)), \quad \quad \chi(t):=V(y(t), z(t)),$$
where $V$ and $W$ are defined in \eqref{Lyapfunz}.
By the Lemma \ref{limitatezzafpi} there exists $t_0\geq 0$ and a positive constant $\overline{C}_i$ for every $i $ in $ \mathcal{E}$, such that for every $t\geq t_0$ and applying the inverse of the function $\varphi_i$ we get 
\begin{equation}\label{roestar}
	x_i^{z}(t)\leq x_i^*, \quad\quad x_i^*:=\varphi_i^{-1}(\overline{C}_i) \quad \forall i \in \mathcal{E}.
\end{equation}
Since $x_i^{z}(t) \geq 0$, \eqref{roestar} implies that if $\lvert x_i(t)-x_i^{z}(t)\rvert \geq 2x_i^*$ for some $t\geq t_0$, then $x_i(t)\geq 2x_i^*$ for $t\geq t_0$. Hence $y_i(t)-y_i^z(t)\geq \chi_i^*$ for all $t\geq t_0$, where $\chi_i^*=\varphi_i(2x_i^*)-\overline{C}_i$. Since $\varphi_i(x_i)$ is a strictly increasing function, one has that
$$\chi_i^*=\varphi_i(2x_i^*)-\overline{C}_i> \varphi_i(x_i^*)-\overline{C}_i=0.$$ Now, let
$$ \zeta^*:=2\vert\mathcal{E}\vert\max\{x_i^*: i \in \mathcal{E}\}, \quad
\chi^*:=\min\{\chi_i^*: i \in \mathcal{E}\}.
$$
and observe that
\begin{equation*}
	\begin{split}
& W(x, z)\leq \lvert\mathcal{E}\rvert\max\{\vert x_i-x_i^{z} \vert : i \in \mathcal{E}\},\\
& V(y, z)\geq \vert y_i-y_i^{z}\vert \ \quad \forall i \in \mathcal{E}.
	\end{split}
\end{equation*}
Therefore, it follows that for any $t\geq t_0$, if $\zeta(t)\geq \zeta^*$, then for some $i' $ in $ \mathcal{E}$ we have that $\vert x_{i'}(t)- x_{i'}^z(t)\vert \geq 2x_{i'}^*$ for $t \geq t_0$. This in turn implies that $\chi(t)\geq\chi_{i'}^*\geq \chi^*$.
Hence,
	\begin{equation}\label{condparametri}
	\zeta(t)\geq \zeta^*\ \Longrightarrow \ \chi(t)\geq\chi^*>0 \quad \forall t\geq t_0.
	\end{equation} 
Moreover by \eqref{roestar} follows that there exists some $\mu>0$ such that 
	\begin{equation*}
	\sum_{i \in \mathcal{E}}\frac{1}{\varphi'_i(x_i^z(t))}\leq \mu \qquad \forall t\geq t_0.
	\end{equation*}
By combining the above with Lemma \ref{gradienteW} one finds that for every $u, t \geq t_0$,
	\begin{equation}\label{stimeintegrali}
	\begin{split}
	&\zeta(t) -\zeta(u) = \int_u^{t}\sum_{i \in \mathcal{E}}\sigma_i\left(\frac{d}{ds}x_i-\frac{d}{ds}x_i^z \right)ds\\
	&\leq \int_u^{t}\nabla_{x}W(x, z)'H(y, z)ds\\
	& + \int_u^{t}\sum_{i \in \mathcal{E}}\frac{\eta}{\varphi'_i(x_i^z(t))}\vert (AF^{(\beta)}(l^z, w^z))_i-(Az)_i\vert ds\\
	& \leq \int_u^{t} \Big(-\varsigma\,\chi(s)+2\lambda\eta\mu\Big)\, ds.
	\end{split}
	\end{equation}
Now, by contradiction, let us assume that $\limsup_{t\to \infty}y_i(t)\geq C_i$ for some $i $ in $ \mathcal{E}$. Since $y_i(t)=\varphi_i(x_i(t))<C_i$ for every $t\geq 0$, this would imply that $\limsup_{t\to \infty}x_i(t)=\infty$. From this follows that the  $\limsup_{t\to \infty}\zeta(t)=\infty$. Then, in particular, the set $\mathcal{T}:=\{t>0:\zeta(t)>\zeta(s)\ \forall\ s<t\}$ is an unbounded union of open intervals with $\lim_{t \in \mathcal{T}, t\to \infty}\zeta(t)=\infty$. This and \eqref{condparametri} imply that there exist a nonnegative constant $t^*\geq t_0$ such that 
	\begin{equation}\label{inequalityimportant}
	\chi(t)\geq \chi^* \quad \forall t \in \mathcal{T}\cap[t^*, \infty).
	\end{equation}
Now defining $\eta^*:=\displaystyle \varsigma\, \chi^*/2\lambda\mu $, for every $\eta<\eta^*$, \eqref{stimeintegrali} and \eqref{inequalityimportant} give
	\begin{equation*}
	\begin{split}
	\zeta(t)-\zeta(u) & \leq  \int_u^{t}\Big(-\varsigma\,\chi(s) +2\lambda\eta\mu\Big) \ ds \\
	 & \leq \int_u^{t} \Big(-\varsigma\, \chi^*+2\lambda\eta\mu\Big) \ ds < 0
	\end{split}
	\end{equation*}
for any $t>u\geq t^*$ such that $t$ and $u$ belong to the same connected component of $\mathcal{T}$. But this contradicts the definition of the set $\mathcal{T}$. Hence, if $\eta<\eta^*$ then $\limsup_{t\to \infty}y_i(t)< C_i$ for any $i $ in $ \mathcal{E}$. Since on every compact time interval $\mathcal{I}\subseteq \mathbb{R}_+$, one has $\sup_{t \in \mathcal{I}}y_i(t)=y_i(\hat{t})<C_i$ for some $\hat{t} $ in $ \mathcal{I}$, the previous implies the claim.
\end{proof}
\begin{lemma}\label{limitesupT}
	There exists constants $K>0$ and $t_1\geq 0$ such that for every initial condition $(z(0), x(0)) $ in $ \mc Z^{\circ}\times\R_+^\mathcal{E}$, $\Vert \tilde{\nabla}_z h(z(t))\Vert \leq K$ for all $t\geq t_1$.
\end{lemma}
\begin{proof}
	From Lemma \ref{limitatezzay}, there exists $T^*, \ups^*>0$ such that $\Vert l(t)\Vert \leq T^*$ and $\Vert w(t)\Vert \leq \ups^*$  for all $t\geq 0$. This fact together with the definition of $F^{(\beta)}(l,w)$ \eqref{bestresponse} imply that $F^{(\beta)}(l(t), w(t))$ belongs to  $\mc Z_\beta^{\circ}$ and $ \tilde{\nabla}_z h(F^{(\beta)}(l(t), w(t)))= - \Phi A'(l(t)+w(t))$. Hence $\Vert \tilde{\nabla}_z h(F^{(\beta)}(l(t), w(t)))\Vert \leq \Vert \Phi \Vert \Vert A' \Vert S^*$, with $S^*=T^*+\ups^*$. This implies the existence of a convex compact $\mathcal{K}\subset \mc Z_\beta^{\circ}$ such that $F^{(\beta)}(l(t), w(t))$ belongs to $\mathcal{K}$ for all $t\geq 0$.
	Define
	\begin{equation*}
	\Delta(t):=\frac{\eta}{1-e^{-\eta t}}\int_{0}^{t}e^{-\eta(t-s)}F^{(\beta)}(l(s), w(s))\,ds.
	\end{equation*}
	Since $\Delta(t)$ is an average of elements of the convex set $\mathcal{K}$, then $\Delta(t)\in \mathcal{K}\ \forall t\geq 0 $. Moreover, $z(t)=e^{-\eta t}z(0)+(1-e^{-\eta t})\Delta(t)$ approaches $\mathcal{K}$, which implies that for large enough $t$, $z(t)$ belongs to
	a closed subset $\mathcal{K}_1$  of $\mc Z_\beta^{\circ}$ that contains $\mathcal{K}$.  Hence, after large enough $t$, say, $t_1$, $\tilde{\nabla}_z h(z(t))$ stays bounded.
	\end{proof}
\begin{lemma}\label{gradienteWpi}
		There exist $\ell>0$ and $t_0\geq 0$ such that for every initial condition $(z(0), x(0)) $ in $ \mc Z^{\circ}\times\R_+^\mathcal{E}$,
		\begin{equation*}
		\tilde{\nabla}_z W(x(t),z(t))'(F^{(\beta)}(l(t), w(t))-z(t))\leq 2\lambda\ell\vert\mathcal{E}\vert \quad \forall t\geq t_0.
		\end{equation*}
\end{lemma}
\begin{proof}
Observe that thanks to Lemma \ref{limitatezzafpi} there exist $t_0\geq 0$ such that $\ell_i:=\sup\{1/\varphi'_i(x_i^z(t)): t\geq t_0\}< +\infty$.
Put $\ell:=\max\{\ell_i:i \in \mathcal{E}\}$. Then, for every path $\gamma $ in $ \Gamma$ and for every $t\geq t_0$, one has
	\begin{equation*}
	\begin{split}
	\left\lvert\frac{\partial W(x, z)}{\partial z_{\gamma}}\right\rvert & = \left\lvert-\sum_{i \in \mathcal{E}}\sigma_i\frac{\partial}{\partial z_{\gamma} }x_i^z\right\rvert\\
	& =\left\lvert\sum_{i \in \mathcal{E}}\sigma_i\frac{\partial}{\partial z_{\gamma} }\varphi_i^{-1}\left( \sum_\gamma A_{i\gamma}z_{\gamma}\right) \right\rvert\\
	&\leq \sum_{i \in \mathcal{E}}A_{i\gamma}\frac{1}{\varphi'_i(x_i^z)} \leq \sum_{i \in \mathcal{E}}A_{i\gamma}\ell_i\leq \ell\vert\mathcal{E}\vert.
	\end{split}
	\end{equation*}	
Therefore,
	\begin{equation*}
	\begin{split}
	2\lambda\ell\vert\mathcal{E}\vert &\geq \sum_\gamma F_\gamma^{(\beta)}(l, w)\left\lvert\frac{\partial W(x, z)}{\partial z_{\gamma}}\right\rvert+\sum_\gamma z_{\gamma}\left\lvert\frac{\partial W(x, z)}{\partial z_{\gamma}}\right\rvert \\
	& \geq \sum_\gamma F_\gamma^{(\beta)}(l, w)\frac{\partial W(x, z)}{\partial z_{\gamma}}-\sum_\gamma z_{\gamma}\frac{\partial W(x, z)}{\partial z _\gamma}\\
	&=\tilde{\nabla}_z  W(x, z )'(F^{(\beta)}(l, w)-z).
	\end{split}
	\end{equation*}	
\end{proof}
We now combine Lemmas \ref{gradienteW} and \ref{gradienteWpi} in order to estimate the behavior in time of $W(x(t), z(t))$.
\begin{lemma} \label{derivatatotaleW} There exist $\ell, L, \eta^*>0$ and $t_0\geq 0$ such that for every initial condition $z(0) $ in $ \mc Z$, $x(0)$ in $ [0, +\infty)^\mathcal{E}$,
	\begin{equation*}
	\begin{split}
	\displaystyle
	& W(x(t),z (t))\leq\\
	&\displaystyle \frac{2\lambda\ell L\eta\vert\mathcal{E}\vert}{\varsigma} +  e^{-\varsigma(t-t_0)/L}\left(W(x(t_0),z (t_0))-\frac{2\lambda\ell L\eta\vert\mathcal{E}\vert}{\varsigma}\right)
	\end{split} 
	\end{equation*}
	for $t\geq t_0$ and $\eta<\eta^*$.
\end{lemma}
\begin{proof}
	Define $\displaystyle\zeta(t):=W(x(t), z(t))$. Note that thanks to Lemmas \ref{limitatezzafpi} and \ref{limitatezzay}, there exist $L>0$, $\eta^*>0$ and $t_0\geq 0$ such that for any $\eta<\eta^*$,
	\begin{equation*}
	\vert x_i(t)-x_i^z(t)\vert\leq L\vert y_i(t)-y_i^z(t) \vert \quad \forall i \in \mathcal{E}, t\geq t_0.
	\end{equation*}
	This involves that
	\begin{equation*}
	V(y(t), z(t))\geq \frac{1}{L}W(x(t), z(t))=\frac{1}{L}\zeta(t) \quad \forall \eta<\eta^*, t\geq t_0. 
	\end{equation*}
	Moreover $W(x,z)$ is a Lipschitz function of $x$ and $z$, while both $x(t)$ and $z(t)$ are Lipschitz on every compact time interval. Therefore $\zeta(t)$ is Lipschitz on every compact time interval and hence absolutely continuous. Thus $d\zeta(t)/dt$ exists for almost every $t\geq 0$, and, thanks to Lemmas \ref{gradienteW} and \ref{gradienteWpi} it satisfies
	\begin{equation*}
	\begin{split}
	\frac{d\zeta(t)}{dt}& = \frac{dW(x(t),z(t))}{dt} \\
	& = \nabla_{x}W(x, z)'H(y, z)+\eta\tilde{\nabla}_z W(x,z)'(F^{(\beta)}(l, w)-z)\\
	&\leq -\varsigma V(y, z)+ 2\lambda\ell\eta\vert\mathcal{E}\vert\leq -\frac{\varsigma\,\zeta(t)}{L}+ 2\lambda\ell\eta\vert\mathcal{E}\vert.
	\end{split}
	\end{equation*}
	Then, integrating both sides we get the claim.
\end{proof}
\subsection{Proof of Theorem \ref{theorem:main}}
We are now in a position to to prove Theorem \ref{theorem:main}. Let us consider the function
\begin{equation}
\Theta: \mc Z\to \mathbb{R}_+, \quad \Theta(z):=\sum_{i \in \mathcal{E}}\int_0^{y_i^z}\Big(\tau_i(s)+\omega_i(s)\Big)\, ds
\end{equation}
and observe that 
\begin{equation}\label{gradienteThetaTollVar}
\tilde{\nabla}\Theta(z)=\Phi A'(l^z+w^z) \qquad \forall z \in \mc Z^{\circ}.
\end{equation}
Note that since $\tau_i(y_i)+\omega_i(y_i)$ is increasing, then the map  $y_i\mapsto\int_0^{y_i^z}\Big(\tau_i(y_i)+\omega_i(y_i)\Big)\,d y_i$ is  convex. Hence, the composition with the linear map $z\mapsto y_i^z=\sum_\gamma A_{i\gamma}z_{\gamma}$ is convex in $z$, which in turn implies convexity of $\Theta$ over $\mc Z$. Since $h(z)$ defined in \eqref{negativentropy} is strictly convex, we obtain strict convexity of $\Theta(z)+h(z)$ on $\mc Z_\beta$. Then, since $\mc Z_\beta$ is a compact and convex set, there exists a unique minimizer
\begin{equation}\label{minimizer}
z^\beta:=\arg\min\{\Theta(z)+h(z): z \in \mc Z_\beta \}. 
\end{equation}
Let now $y^{(\omega,\beta)}:=y^{z^\beta}$. Then, the following  result holds true.
\begin{lemma}\label{Dim1parteTeorema}
The perturbed equilibrium flow $y^{(\omega,\beta)} $ in $ \mathcal{F}$ is such that
	\begin{equation*}
	\lim_{\beta\to+\infty}y^{(\omega,\beta)}=y^{(w)}. 
	\end{equation*}
\end{lemma}
\begin{proof} Since $\{Az^\beta\} \subseteq A\mc Z$, and $A\overline{\mc Z}$
	is compact, there exists a converging subsequence $\{Az^{\beta_k} : k\in \mathbb{N}\}$. Let us denote by
	$\hat y := \lim_k Az^{\beta_k} $ in $ A\overline{\mc Z}$ its limit and choose some $\hat z $ in $ \overline{\mc Z}$ such that $\hat y = A\hat z$. Notice that since $$\sup\{\tau_i(y_i^z)+\omega_i(y_i^z): z\in \mc Z_\beta\} < +\infty\,,\qquad \forall i \in \mathcal{E}\,,$$ the differentiability of $h$ in the interior set $\mc Z_\beta^{\circ}$ of $\mc Z_\beta$ implies that the
	minimizer in \eqref{minimizer} belongs to $\mc Z^\circ_\beta$. As a consequence, one finds that necessarily 
	$$\tilde{\nabla}_z h(z^{\beta_k})=-\Phi A'(\tau(Az^{\beta_k})+\omega(Az^{\beta_k})),$$ which successively implies that $F^{\beta_k}(\tau(Az^{\beta_k}), \omega(Az^{\beta_k}))=z^{\beta_k}$.
	Then, using \eqref{generalbestresponse}, one finds that
	\begin{equation}\label{catenadisug}
	\begin{aligned}
	&(Az^{\beta_k})'(\tau(Az^{\beta_k})+\omega(Az^{\beta_k}))+ h_{\beta_k}(z^{\beta_k})\\\leq  &(Az^{\beta_k})'(\tau(Az^{\beta_k})+\omega(Az^{\beta_k}))+h_{\beta_k}(\alpha),
	\end{aligned}
	\end{equation}
	for all $\alpha $ in $ \mc Z_{\beta_k}$. Now, fix any $z $ in $ \mc Z$. Since $\mc Z_{\beta}\to \overline{\mc Z}$ as $\beta\to+\infty$,\footnote{Here, $\overline{\mc Z}$ stands for the closure of $\mc Z$ and the convergence $\mc Z_{\beta}\to \overline{\mc Z}$ is meant to hold true with respect to the Hausdorff metric.} then there exists a
	sequence $\{\tilde{z}^k\}$ such that $\tilde{z}^k$ belongs to $\mc Z_{\beta_k}$ for all $k$ and $\lim_k \tilde{z}^k= z$. Hence, taking $\alpha=\tilde{z}^k $ in \eqref{catenadisug}
	and passing to the limit as $k$ grows large, one finds that
	$$ \hat z'A'(\tau(\hat y)+\omega(\hat y))\leq z'A'(\tau(\hat y)+\omega(\hat y)) \quad \forall\ z \in \mc Z.$$
	In turn, the above can be easily shown to be equivalent to the characterization \eqref{Wardrop} of Wardrop equilibria. From the uniqueness of the Wardrop equilibrium, it
	follows that necessarily $\hat y= y^{(w)}$. Then the claim follows from the arbitrariness of the accumulation point $\hat y$, hence $y^{(\omega,\beta)}\to y^{(w)}$.
\end{proof}
\vspace{0.3cm} 
We now estimate the time derivative of $\Theta(z)+h(z)$ along trajectories of our dynamical system. Towards this goal,  define
\begin{equation*}
\ba{rcl}
\Psi(t)&:=&\Theta(z(t))+h(z(t)),\\[7pt]
\psi(t)&:=&\Phi A'(l^z(t)+w^z(t))+\tilde{\nabla}_z h(z(t))\,.
\ea
\end{equation*}
Then, using \eqref{gradienteThetaTollVar}, we get
\begin{equation}\label{derivataGamma}
\ba{rcl}
\dot{\Psi}(t)
&=&\Big(\tilde{\nabla}_z\Theta+ \tilde{\nabla}h(z)\Big)\dot{z}\\[7pt]
&=&\eta\psi(t)'(F^{(\beta)}(l(t), w(t))-z(t)) \\[7pt]
&=& \eta\psi(t)'(F^{(\beta)}(l^z(t), w^z(t))-z(t))\\[7pt] 
&&+  \eta\psi(t)'(F^{(\beta)}(l(t), w(t))-F^{(\beta)}(l^z(t), w^z(t)))\,.
\ea
\end{equation}
By Lemma \ref{derivatatotaleW}, there exist $t_2\geq 0, \eta^*>0$ and $M_1>0$ such that  $W(x(t),z(t))\leq \eta M_1$ for all $\eta<\eta^*$ and  $t\geq t_2$. From the definition of $W$ it follows that $W(x, z)\geq \Vert x-x^z\Vert_1/ \vert \mathcal{E} \vert$ for all $x, z$. Moreover, the properties of $\varphi$ imply that $\Vert y-y^z\Vert_1\leq \overline{L}\Vert x-x^z\Vert_1$ for all $y$, $z$, and $\overline{L}:=\max\{\varphi'_i(0): i \in \mathcal{E}\}$. Combining all these relationships we get that there exists $M>0$ such that, for every $\eta<\eta^*$,
\begin{equation}\label{normaf-fpi}
\Vert y(t)-y^z(t)\Vert \leq \eta M \qquad \forall t\geq t_2,
\end{equation}
where $M=\vert \mathcal{E} \vert M_1\overline{L} $.
Thanks to the differentiability of $F^{(\beta)}$ on $\R_+^{\mc E}\times\R_+^{\mc E}$ and the boundedness of both  $y^z(t)$ and $y(t)$ 
one gets that 
\begin{equation*}
\Vert F^{(\beta)}(l(t), w(t))- F^{(\beta)}(l^z(t), w^z(t))\Vert \leq K_1 \eta
\end{equation*}
for some positive constant $K_1$, $\eta<\eta^*$ and large enough  $t$.
Since Lemmas \ref{limitatezzafpi} and \ref{limitesupT} guarantee that 
that $l^z(t)$, $w^z(t)$ and $\tilde{\nabla}_z h(z(t))$ are eventually bounded, 
there exists a positive constant $K_2$ such that $\Vert \psi(t)\Vert\leq K_2$ for $t$ large enough. This implies that the second addend in the last line of \eqref{derivataGamma} can be bounded as 
\begin{equation}\label{secondaddend}
\eta\psi(t)'(F^{(\beta)}(l(t), w(t))-F^{(\beta)}(l^z(t), w^z(t))) \leq K \eta^2
\end{equation}
where $K=K_1K_2$, for all $\eta<\eta^*$ and $t\geq t_3$ for some sufficiently large but finite value of $t_3$. Now, observe that 
\begin{equation*}
\Phi A'(l^z(t)+w^z(t)))=-\tilde{\nabla}_z h(F^{(\beta)}(l^z(t), w^z(t)))
\end{equation*}
for every $z $ in $ \mc Z$, 
so that the first addend in the last line of \eqref{derivataGamma} may be rewritten as
\begin{equation}\label{firstaddend}
\psi(t)'(F^{(\beta)}(l^z(t), w^z(t))-z(t))=-\Upsilon(z(t)),
\end{equation}
where
\begin{equation*}
\begin{split}
\Upsilon(z(t))=&\Big(\tilde{\nabla}_z h(F^{(\beta)}(l^z(t), w^z(t)))-\tilde{\nabla}_z h(z(t))\Big)' \\
& \cdot (F^{(\beta)}(l^z(t), w^z(t))-z(t)).
\end{split}
\end{equation*}
It follows from \eqref{derivataGamma}, \eqref{secondaddend}, and \eqref{firstaddend} that for $\eta<\eta^*$ and $t\geq t_3$,
\begin{equation}\label{derivataGamma2}
\dot{\Psi}(t)\leq -\eta\Upsilon(z(t))+K\eta^2.
\end{equation}
From the strict convexity of $h(z)$ on $\mc Z_\beta$, $\Upsilon(z(t))\geq 0$ for every $z$, with equality if and only if $z=z^\beta$. Now, put
$$
\begin{array}{ll}
\bar\delta(r)=\\
\begin{cases}
\sup\{\Vert y^z-y^{(\omega,\beta)}\Vert: \Upsilon(z)\leq Kr\} +Kr & \text{if} \quad 0\leq r<\eta^*,\\
\tilde{C}\sqrt{\vert\mathcal{E}\vert} & \text{if} \quad r\geq \eta^*,
\end{cases}
\end{array}
$$
where $\tilde{C}:=\max\{1, \tilde{C}_i : i \in \mathcal{E}\}$, with $\tilde{C}_i$ as defined in Lemma \ref{limitatezzay}. It can be proved that $\bar\delta(r)$ is nondecreasing, right-continuous, and such that $\lim_{\eta\to 0}\bar\delta(\eta)=\bar\delta(0)=0.$ Then, \eqref{normaf-fpi} and \eqref{derivataGamma2} imply that for $\eta<\eta^*$,
\begin{equation}\label{limsupnorma}
\limsup_{t\to \infty}\Vert y(t)-y^{(\omega,\beta)}\Vert \leq \bar\delta(\eta).
\end{equation}
Note that since $y(t) $ in $ [0, \tilde{C}]^{\mathcal{E}}$ and $y^{(\omega,\beta)} $ in $ AZ\subseteq [0, 1]^\mathcal{E}$ then $\vert y_i(t)-y_i^{(\beta)}\vert\leq \max\{\tilde{C}_i, 1 \} \leq \tilde{C}$ for all $i $ in $ \mathcal{E}$ and hence $\Vert y(t)-y^{(\omega,\beta)}\Vert^2\leq \vert\mathcal{E}\vert \tilde{C}^2$. Then \eqref{limsupnorma} also holds for $\eta\geq \eta^*$, since in that range $\bar\delta(r)=\tilde{C}\sqrt{\vert\mathcal{E}\vert}$. This together with Lemma \ref{Dim1parteTeorema} conclude the proof. \qed

{
\section {Possible extensions of the results}\label{sect:extensions}

As discussed, the framework and results presented in the previous sections have arguably two major limitations: the assumption that there is a single origin/destination pair and the assumption that the link flow-density functions are strictly increasing. In this section we briefly discuss possible extensions of our results that include relaxations of these two assumptions. 

First, it is possible to extend our results to the case of multiple origin-destination pairs as follows. 
Let $\{(o_k,d_k)\}_{k\in\mc K}$ be a set of origin-destination pairs where $o_k\ne d_k$ in $\mc V$ for each $k$ in $\mc K$. 
Let $\lambda$ in $\R_+^{\mc K}$ be a vector of associated throughputs $$\nu=\sum_{k\in\mc K}\lambda_k\left(\delta^{(\theta_{o_k})}
-\delta^{(\kappa_{d_k})}\right)\,,\qquad\nu^+=[\nu]\,\qquad\nu^-=[\nu]_-\,.$$
Let $\Gamma_k$ be the set of $(o_k,d_k)$-paths and $A^{(k)}$ in $\{0,1\}^{\mc E\times\Gamma_k}$ the link-path incidence matrix. Let 
$\Gamma=\cup_{k\in\mc K}\Gamma_k$ and $A$ in $\{0,1\}^{\mc E\times\Gamma}$ be the link-path incidence matrix. 
Let
$$\mc S_{\lambda}=\left\{z\in\R_+^{\Gamma}:\,\sum\nolimits_{\gamma\in\Gamma_k}z_{\gamma}=\lambda_k\right\}$$
For every $z$ in $\mc Z_{\lambda}$, $y^z=Az$ is an equilibrium flow vector satisfying 
$By^z=\nu$. Define $G(z)$ as in \eqref{Rij=Gj} and extend \eqref{mass-cons} and \eqref{H} as 
\be\label{mass-cons-gen}\dot x_i(t)=\nu^+_i+ \sum_{j\in\mc E}R_{ji}(t)y_j(t)-y_i(t)\,,\ee
and
\begin{equation}\label{H-gen}
	H_i(y, z):= G_{i}(z)\bigg(\nu_i^++\sum_{j: \kappa_j=\theta_i}y_j\bigg)-y_i\,,\qquad i\in\mc E\,.
\end{equation}
respectively. 
Then, all the results carry over with the notion of  Wardrop equilibrium defined as in \cite[Sect.~2.1]{Patriksson} and the min-cut feasibility condition (cf.~\cite{Robust3})
$$ \sum_{i\in\mc U}\nu_i<\sum_{\substack{i\in\mc E\,:\\\theta_i\in\mc U,\,\kappa_i\notin\mc U}}C_i\,,\qquad \forall\mc U\subseteq\mc V\,.$$

Notice that the extension illustrated above allows one for considering multiple origin-destination pairs. However, it considers a physical dynamics of the traffic flows with a single aggregate commodity, while it keeps the commodities separated as far as the route decision dynamics are concerned. An alternative approach could entail a multicommodity model also of the  physical dynamics of the traffic flows. However, such multicommodity dynamical flow networks would lose fundamental monotonicity properties (cf.~\cite{Nilsson}) that enable, in particular, the proof of Lemma \ref{lemma:l1} as presented in this paper. This means that, in order to generalize the results of this paper with a multicommodity physical dynamics of the traffic flows, one should be able to find different ways to guarantee their global exponential stability. 

Finally, as mentioned in Remark \ref{remark:increasing}, the fact that the flow-density functions are strictly increasing limits the applicability of the results in this paper in road traffic network applications to the so-called free-flow region. One possible approach to extend the setting outside such free-flow region consists in modeling the physical dynamics of the traffic flows with monotone non-FIFO versions of the Cell Transmission Model \cite{Daganzo} as proposed and analysed, e.g., in \cite{Lovisari.ea:2014}, thus keeping monotonicity  and contractivity properties of the physical flow dynamics. The difficulty in this case comes from the fact that the outflow from and the latency on a cell would depend on the densities both on that cell and on the ones immediately downstream, thus making one lose separability of the latency functions. Such an approach may possibly be pursued using techniques developed in the context of traffic assignment problems with non-separable cost functions, see, e.g., \cite{Dafermos1}--\cite{Dafermos3} and \cite[Section 2.5]{Patriksson}. 
}

\section{Numerical simulations}\label{section5}
In this section, we present a numerical study comparing both the asymptotic and transient performance of multiscale transportation networks controlled by dynamic feedback marginal cost tolls as in \eqref{marginalcost} and precomputed constant marginal cost tolls as in \eqref{constolls}.

{For the network topology of Figure \ref{graphtopology} and for several 
%
values of the parameter $\eta$,} in all of our simulations we found that dynamic feedback marginal cost tolls outperform the constant marginal ones. Specifically:
\begin{itemize}
	\item concerning the transient convergence, it appears that the time needed to reach the perturbed equilibrium associated to the dynamic feedback marginal cost tolls is lower than the time to reach the perturbed equilibrium associated to the constant marginal cost ones. 
	\item as the uncertainty parameter $\beta$ of the route choice goes to infinity
	the perturbed equilibrium associated to dynamic feedback marginal cost tolls asymptotically converges to the social optimum flow faster than the one associated to the constant marginal cost tolls.
\end{itemize}

We illustrate these findings in the following simple case:
\begin{itemize}
	\item network topology $\mathcal{G}$ as in Figure \ref{graphtopology};
	\item flow-density function as in \eqref{ex:phi} 
	and corresponding latency function as in \eqref{ex:tau}, with capacity $C_i=2$ for every link $i$ in $\mc E$; 
	\item $F^{(\beta)}$ as in \eqref{bestresponse}, $\eta=0.1$, $G$ as in \eqref{localchoice} and $\lambda=1$;
	\item initial conditions: $z_{\gamma^{(1)}}(0)=1/2$, $z_{\gamma^{(2)}}(0)=1/6$, $z_{\gamma^{(3)}}(0)= 1/3$ $x_{i_1}(0)=4$, $x_{i_2}(0)=2$, $x_{i_3}(0)=3$, $x_{i_4}(0)= 1$, $x_{i_5}(0)=5$.
\end{itemize}

\begin{figure}
	\centering
	\begin{tikzpicture}
	[scale=1.1,auto=left,every node/.style={circle,draw=black!90,scale=.5,fill=blue!40,minimum width=1cm}]
	\node (n1) at (0,0){\Large{o}};  
	\node (n2) at (2,1){\Large{a}}; 
	\node (n3) at (2,-1){\Large{b}}; 
	\node (n4) at (4,0){\Large{d}}; 
	\node [scale=0.8, auto=center,fill=none,draw=none] (n0) at (-0.8,0){};
	\node [scale=0.8, auto=center,fill=none,draw=none] (n5) at (4.8,0){};
	\foreach \from/\to in
	{n0/n1,n1/n2,n1/n3,n2/n3,n2/n4,n3/n4,n4/n5}
	\draw [-latex, right] (\from) to (\to); 
	\node [scale=2,fill=none,draw=none] (n5) at (1,0.7){$i_1$};  
	\node [scale=2,fill=none,draw=none] (n6) at (1,-0.7){$i_2$}; 
	\node [scale=2,fill=none,draw=none] (n7) at (1.85,0){$i_3$}; 
	\node [scale=2,fill=none,draw=none] (n8) at (3,0.7){$i_4$};  
	\node [scale=2,fill=none,draw=none] (n9) at (3,-0.7){$i_5$}; 
	\node [scale=1.5,fill=none,draw=none] (n10) at (-0.4,0.2){$1$};
	\node [scale=1.5,fill=none,draw=none] (n14) at (4.4, 0.2){$1$};
	\node [scale=1,fill=none,draw=none] (n11) at (4.1,0){};
	\node [scale=1,fill=none,draw=none] (n12) at (2,1.1){};
	\node [scale=1,fill=none,draw=none] (n13) at (2,-1.1){}; 
	\end{tikzpicture}  
	\caption{\label{graphtopology} The graph topology used for the simulations.}  
\end{figure}
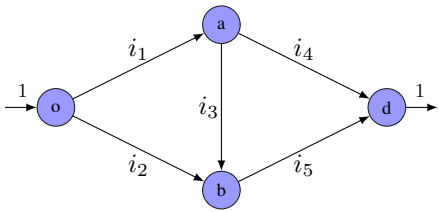
Having settled a time horizon $T=350$, Figure~\ref{DifferenzaDynAccopp}  displays  the $l_1$ distance and the latency loss of $y^{(\omega,\beta)}(T)$ from the system optimum {$y^{*}=(1/2,1/2, 0,1/2, 1/2)$}, for different values of the uncertainty parameter $\beta$. This is done both considering \eqref{marginalcost} and the constant marginal tolls \eqref{constolls}.
Note that while our theoretical results guarantee that $y^{(\omega,\beta)}(T)$ converges to $y^{*}$ only in the double limit of large $T$ (asymptotically in time) and large $\beta$ (vanishing noise), in our numerical examples convergence is practically observed already for relatively small values of $\beta$. 
\begin{figure}
	\centering
	\includegraphics[scale=0.3]{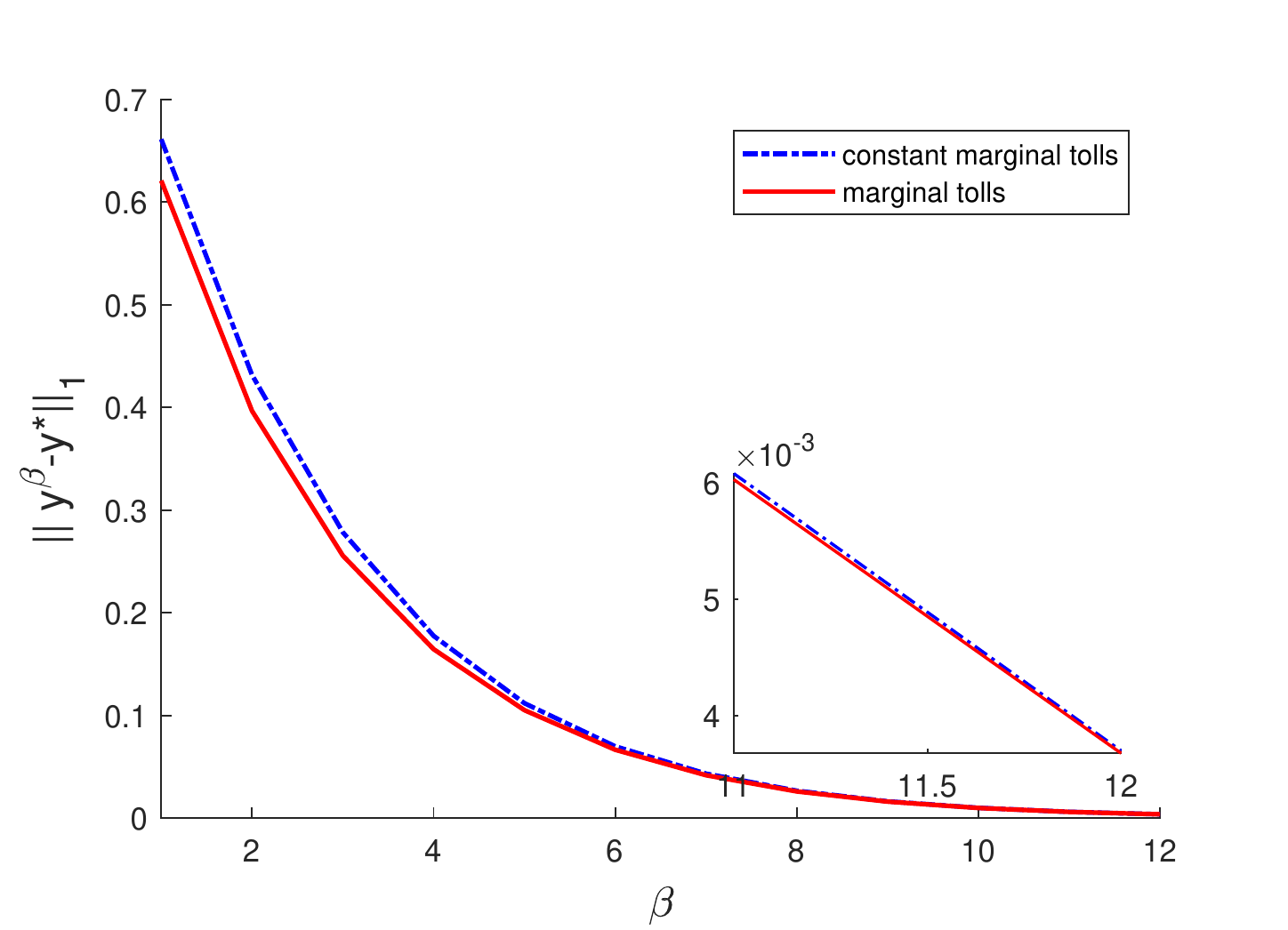}\includegraphics[scale=0.3]{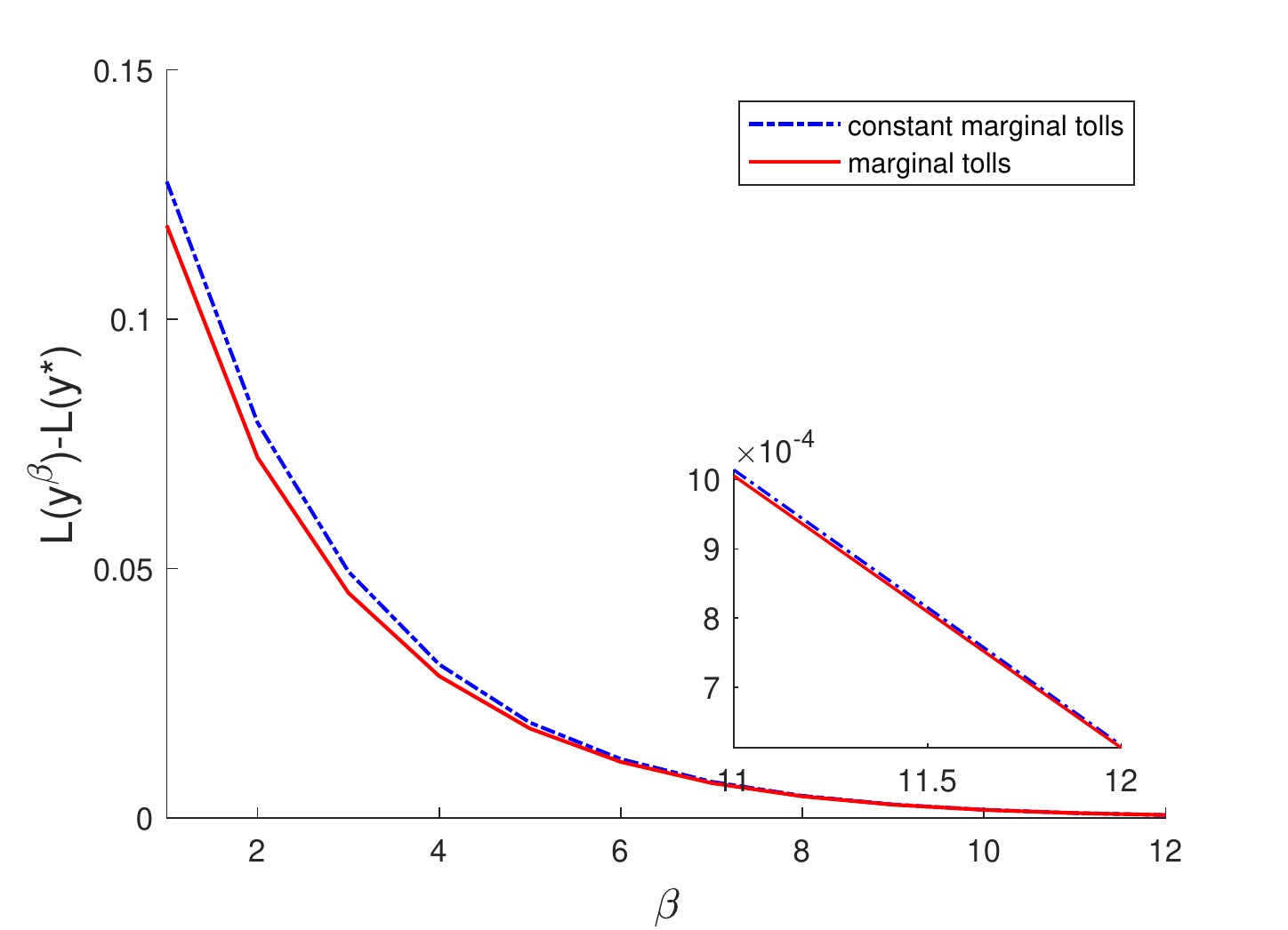}
	\caption{\label{DifferenzaDynAccopp} Plot of $\Vert y^{(\omega,\beta)}(T) - y^* \Vert_1$ and  $\mathcal{L}(y^{(\omega,\beta)}(T))-\mathcal{L}(y^{*})$ for decentralised marginal and constant marginal tolls .}
\end{figure}
Our simulations also suggest that convergence of $y^{(\omega,\beta)}(T)$ to the system optimum is faster for the feedback marginal cost tolls \eqref{marginalcost} than for the fixed marginal cost \eqref{constolls}. Hence, in addition to variations of network's parameters and exogenous loads,  feedback marginal cost tolls appear to be more robust than their constant counterparts also when it comes to noise. 
%
\subsection{Effect of information delays}
In this subsection, we study the effects of delays in the global information of the slow scale dynamics \eqref{evolpi} on the system \eqref{sistemaccoppiato}. Considering at first the case of marginal cost tolls, we fix a time-delay $\phi$ so that 
the cost perceived by each user crossing a link $i $ in $ \mathcal{E}$ is  $l_i(t-\phi)+ w_i(t-\phi)$. Fixing the uncertainty parameter $\beta$ and varying $\phi$, we observe how  the time-evolution of the density $x(t)$ is changed and how the correspondent flow $y$ approximates the social optimum flow $y^*(\lambda)$ with $\lambda=1$.
For that, we consider the graph topology as in Figure \ref{graphtopology} and the same parameters as before. Then, fixing $\beta=5$, we numerically compute the trajectory $x(t)$ for different values of the delay $\phi$ as shown in Figure \ref{FigTraiett}.
\begin{figure}
	\centering%
	\subfigure[\label{fig1}]%
	{\includegraphics[scale=0.267]{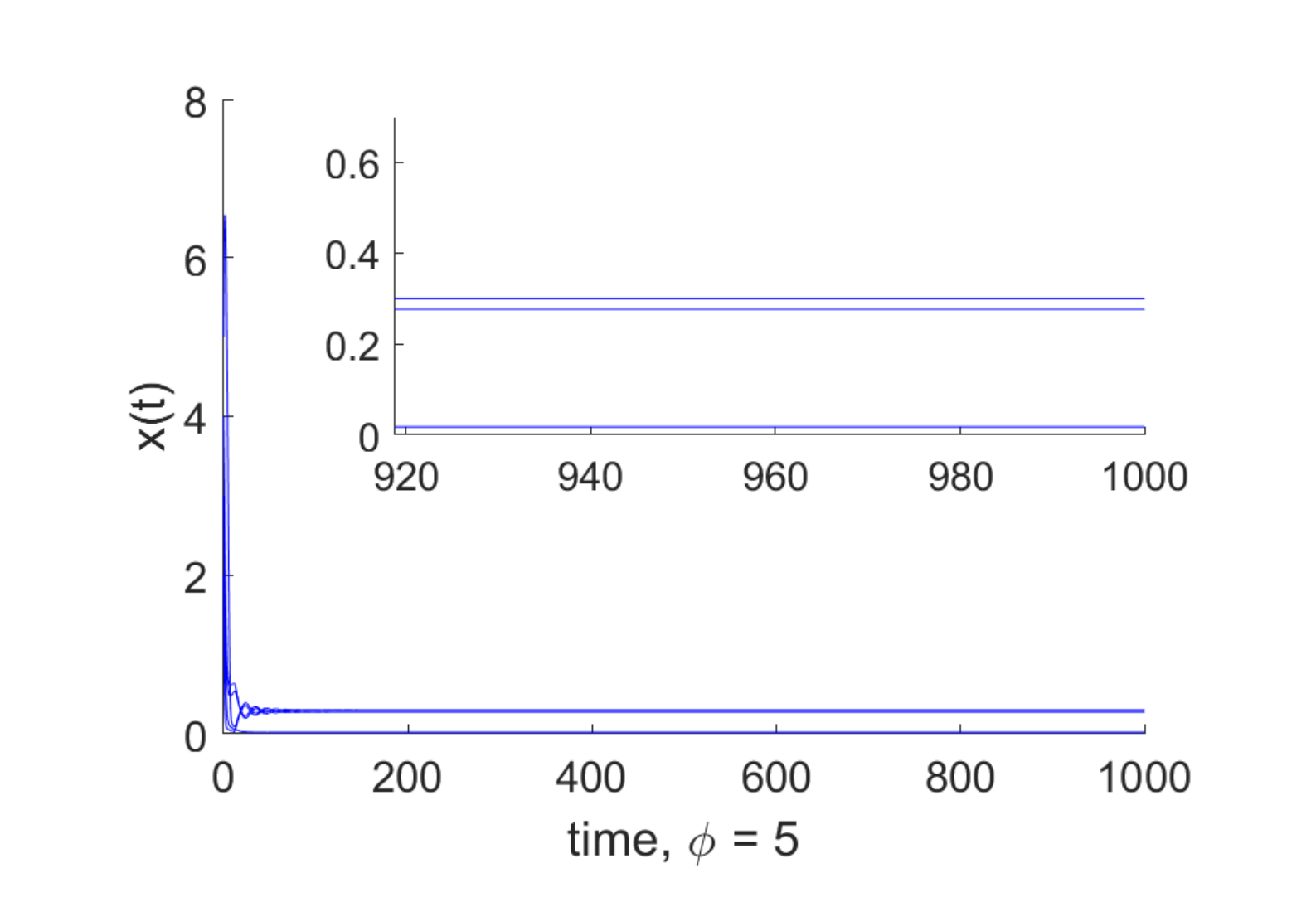}} 
	\subfigure[\label{fig2}]%
	{\includegraphics[scale=0.267]{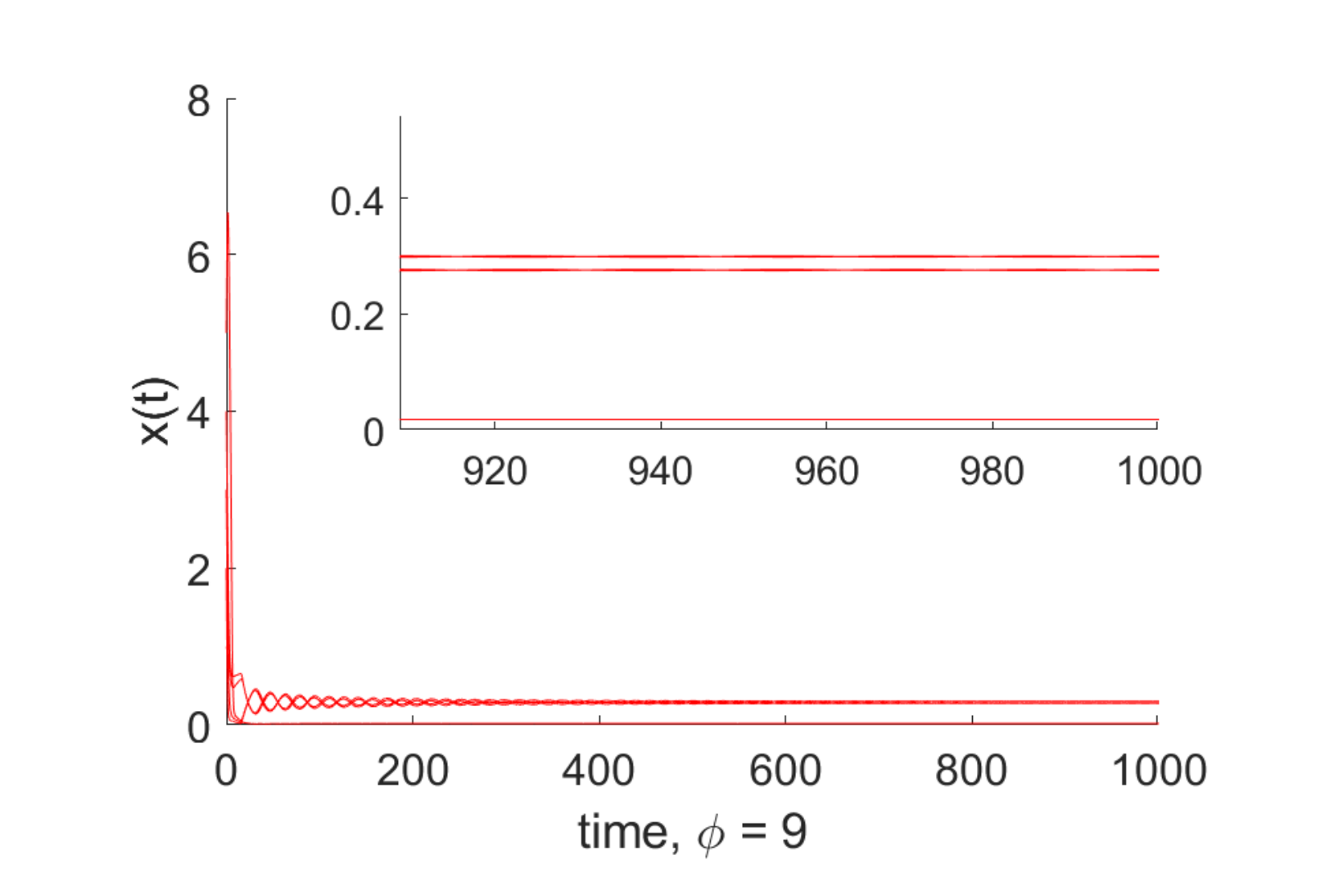}}
	\subfigure[\label{fig3}]%
	{\includegraphics[scale=0.257]{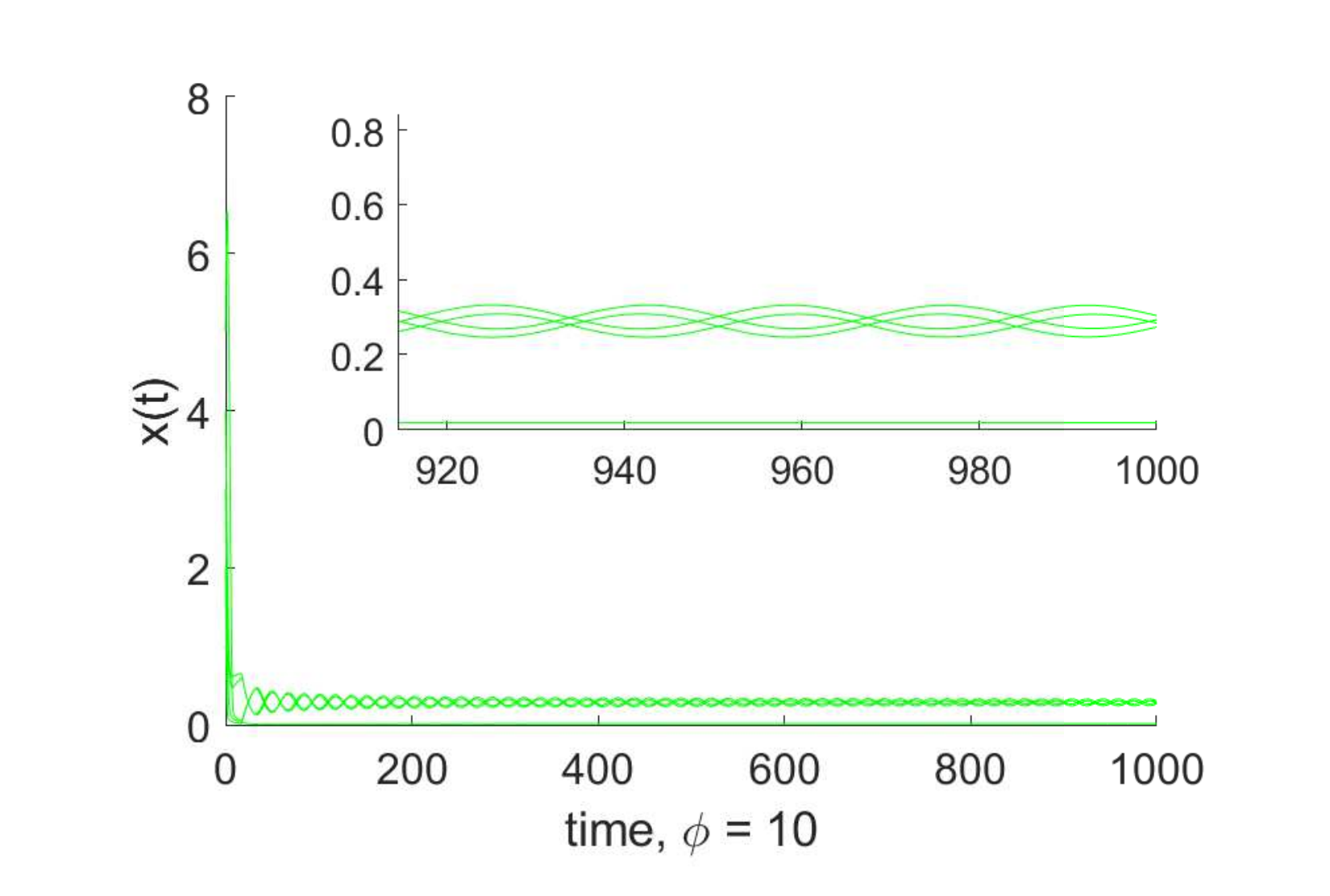}}
	\subfigure[\label{fig4}]%
	{\includegraphics[scale=0.267]{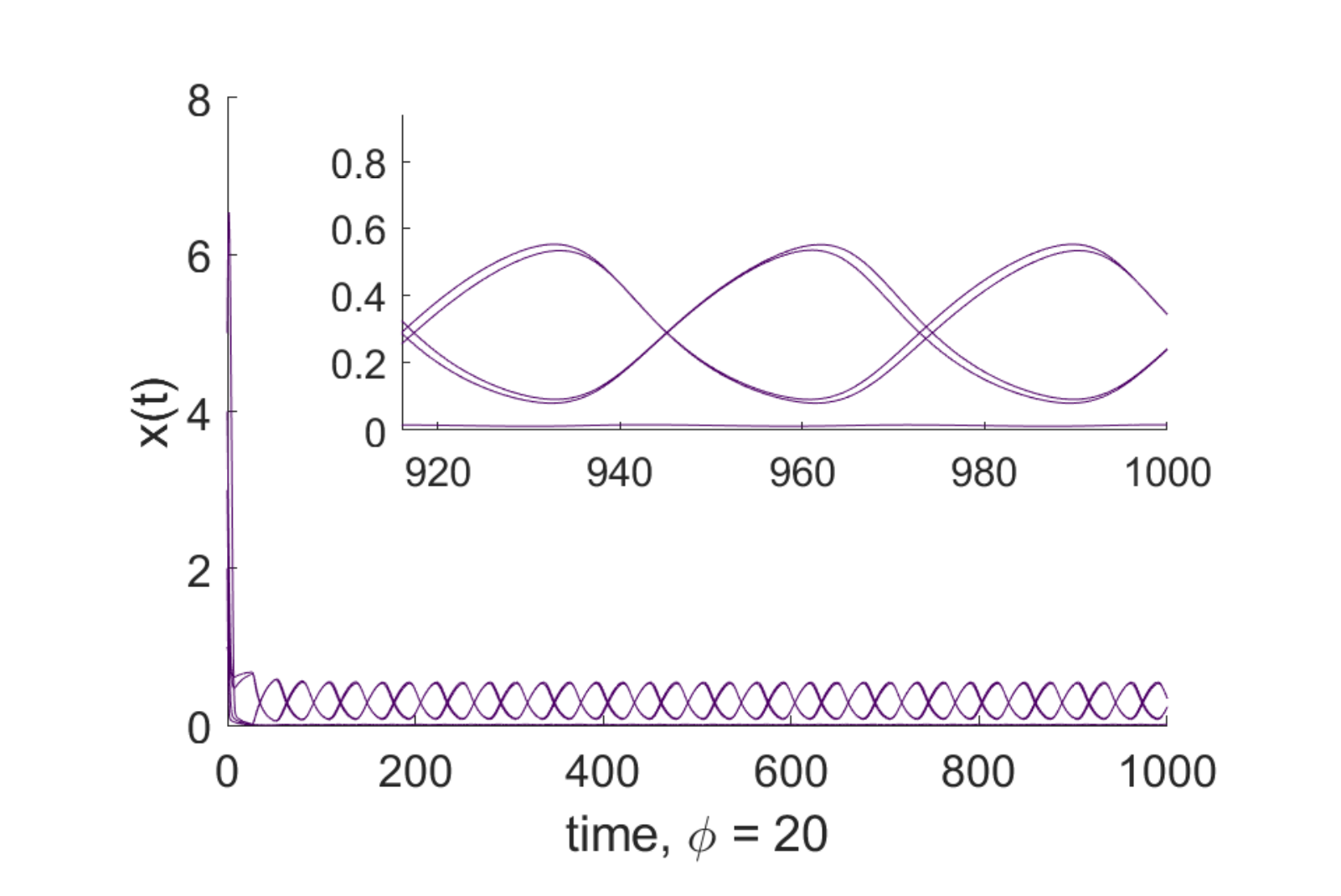}}
	\caption{\label{FigTraiett} The density vector trajectory $x(t)$ for two different values of  the information delay, $\phi=10$ and $\phi=20$.} 
\end{figure}
In Figures~\ref{fig1} and \ref{fig2} we can note that the density vector $x(t)$ converges to an equilibrium. By numerical simulations one gets that $\phi=9$ is the largest value for which one has convergence (see Figure \ref{fig2}). In fact, for $\phi> 9$ one witnesses a phase transition of the system, with the emergence of an oscillatory behavior. 
We can also note in Figures~\ref{fig3} and \ref{fig4} that the larger $\phi$ is, the larger the oscillation amplitude and phase are. A similar situation can be observed in the plot of the $l_1$-distance of $y$ from $y^*$ in Figure \ref{normaNuovo}, for the same value of $\phi$ used in Figure~\ref{FigTraiett}.
\begin{figure}
	\centering	
	\includegraphics[scale=0.44]{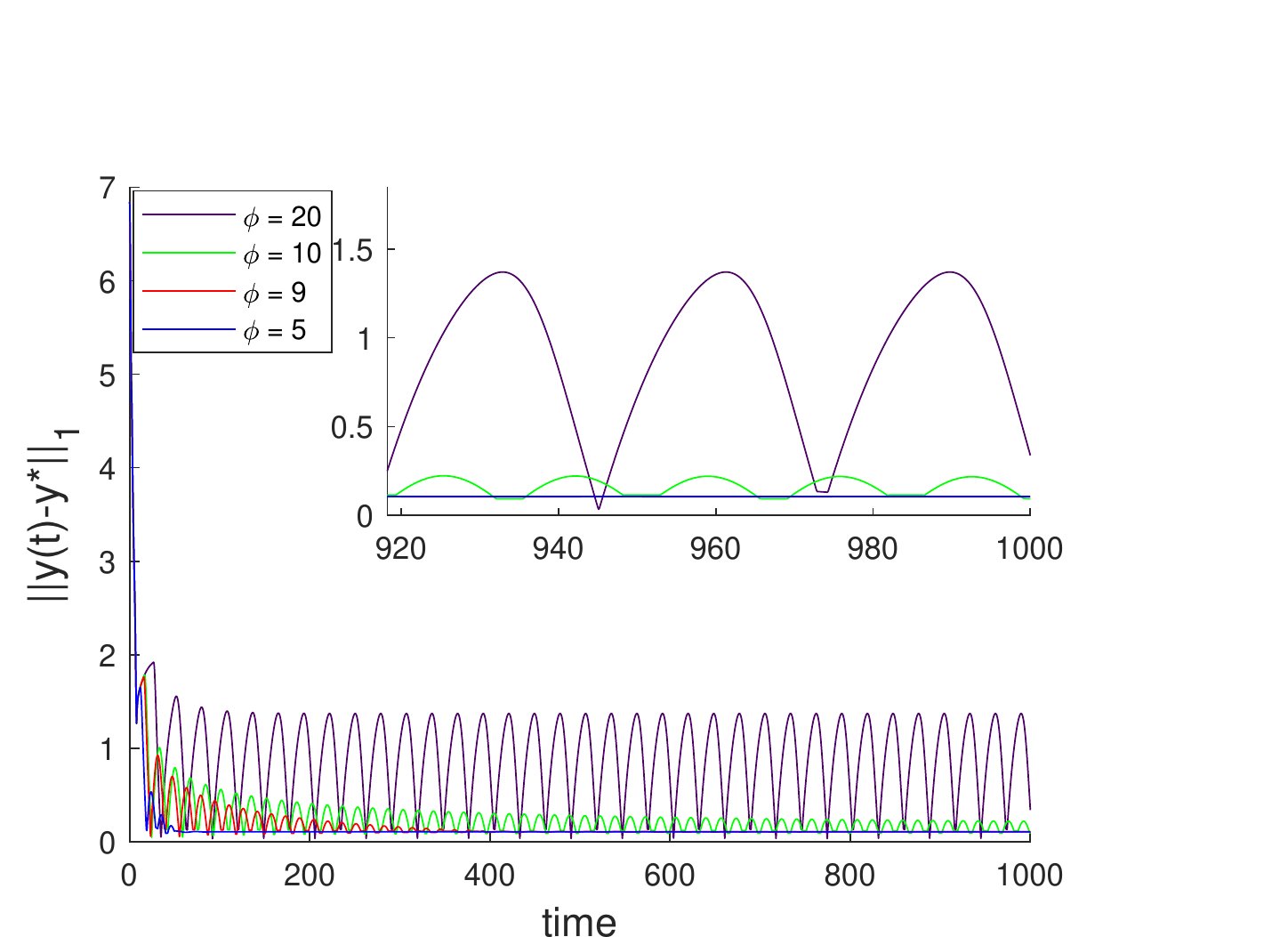}
	\caption{\label{normaNuovo} Plot of $\Vert y(t) - y^* \Vert_1$ for different values of the delay $\phi$.} 
\end{figure}

Consider now the case of constant marginal cost tolls \eqref{constolls}. Let $\phi$ be the time delay as before and $\tau_i(y_i(t-\phi))+w_i^*$ the cost perceived by each user crossing a link $i$ in $\mathcal{E}$. Still using the graph topology as in Figure \ref{graphtopology} and fixing $\beta=5$ we numerically compute the trajectory of the density vector $x(t)$ and the  $l_1$-distance of the corresponding flow  vector $y(t)$ from the social optimal $y^*$.  We perform this for the same values of time delay $\phi$ used before.
\begin{figure}
	\centering%
	\subfigure[\label{fig11}]%
	{\includegraphics[scale=0.262]{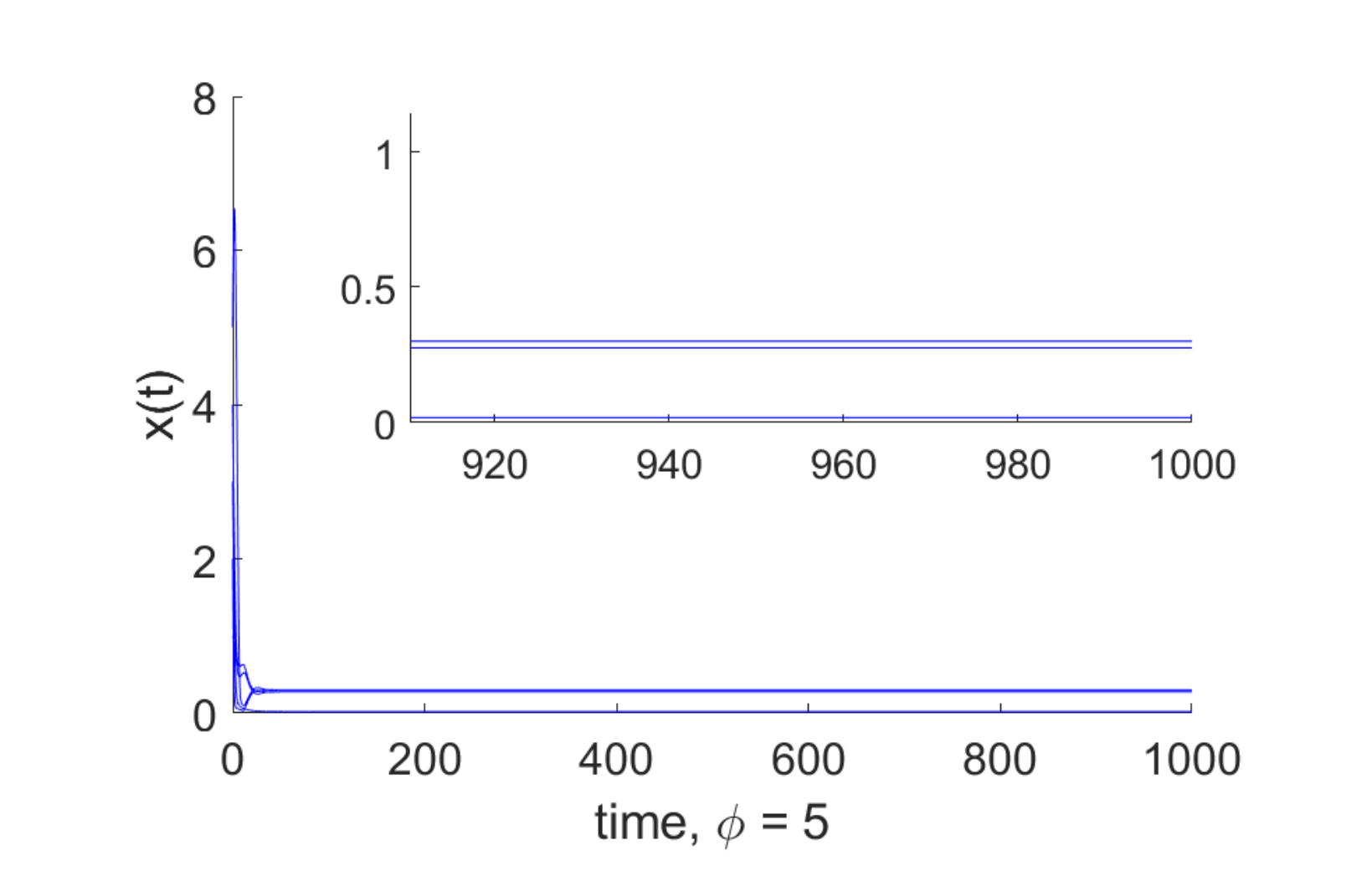}}
	\subfigure[\label{fig22}]%
	{\includegraphics[scale=0.261]{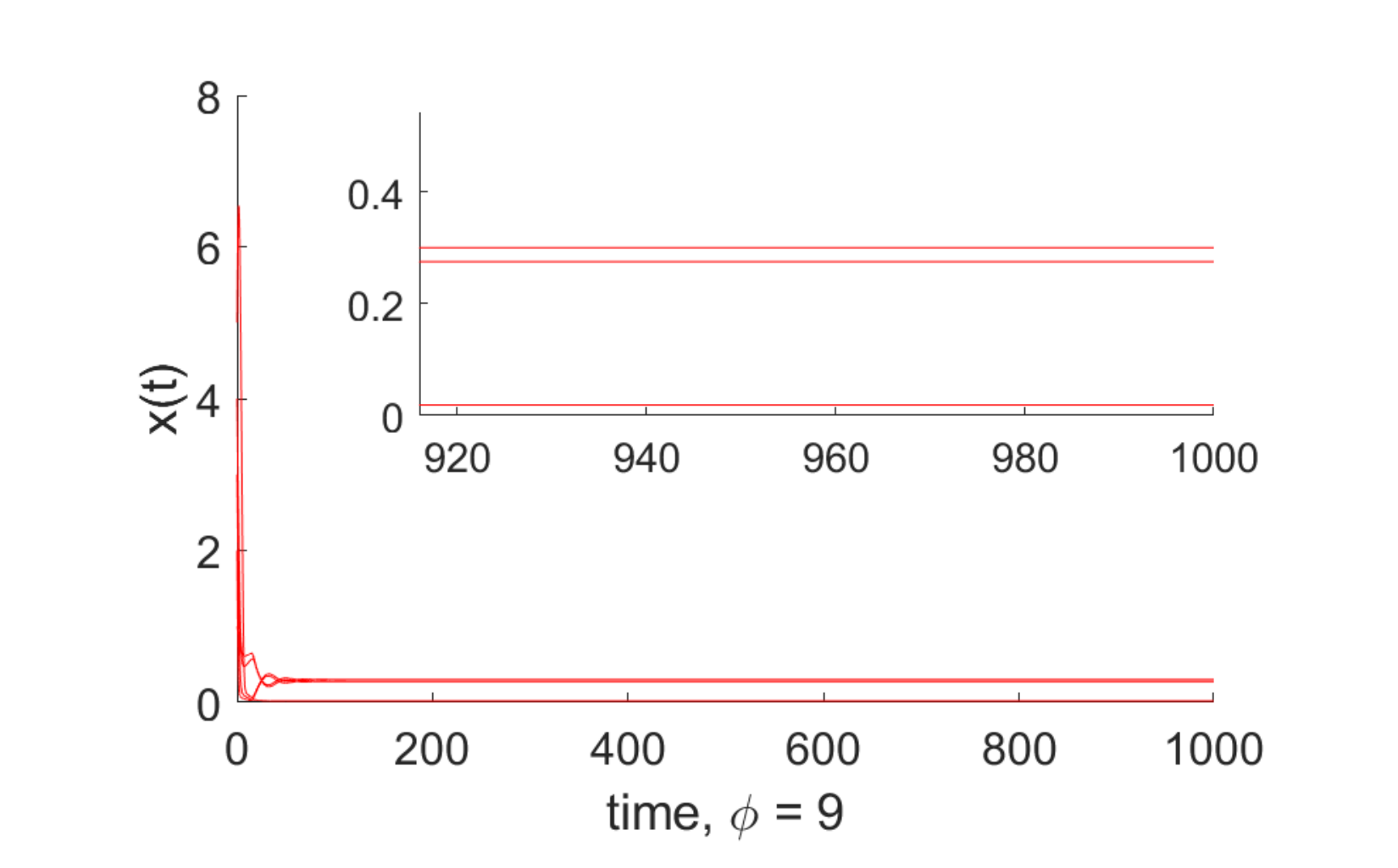}}
	\subfigure[\label{fig33}]%
	{\includegraphics[scale=0.261]{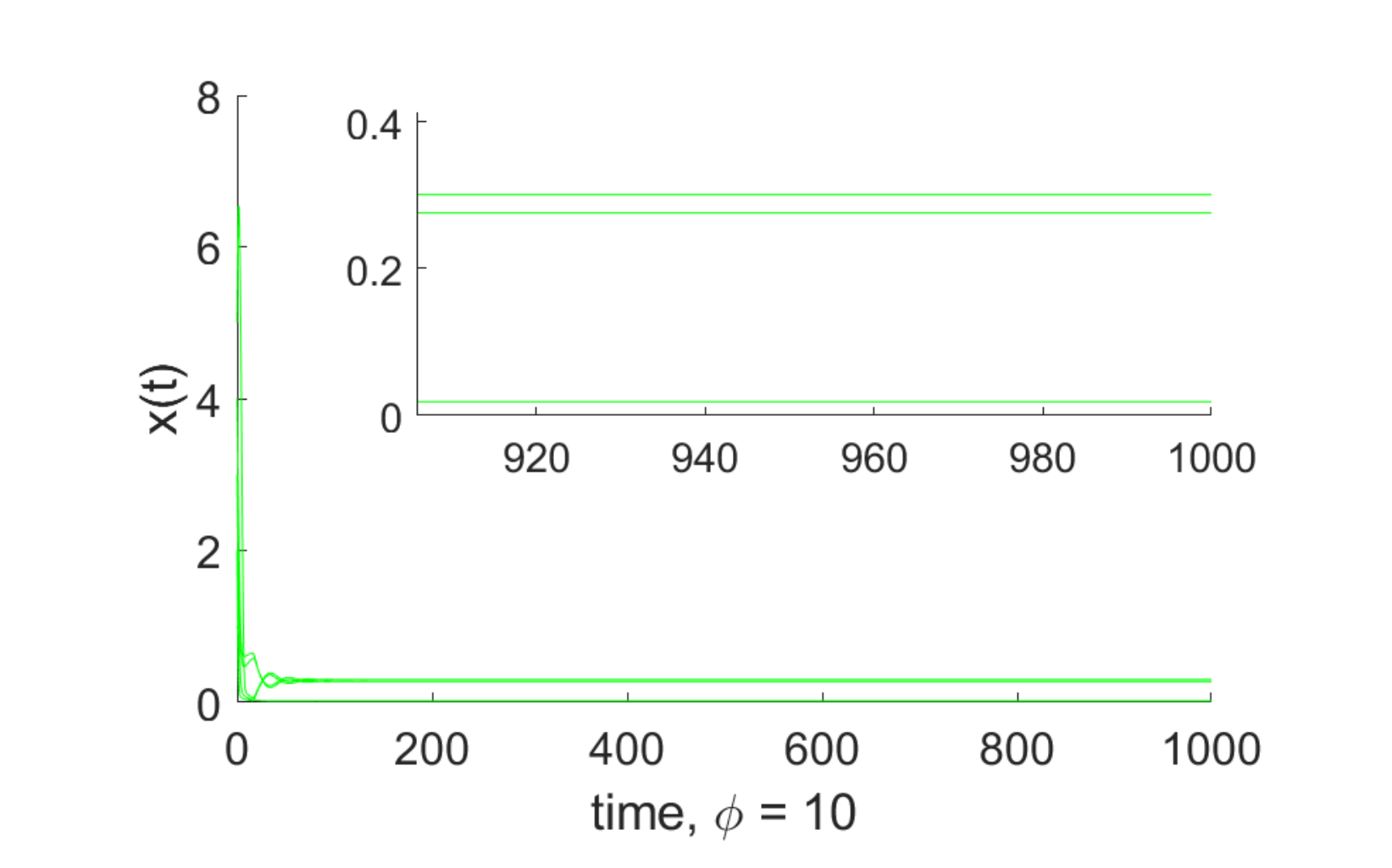}}
	\subfigure[\label{fig44}]%
	{\includegraphics[scale=0.264]{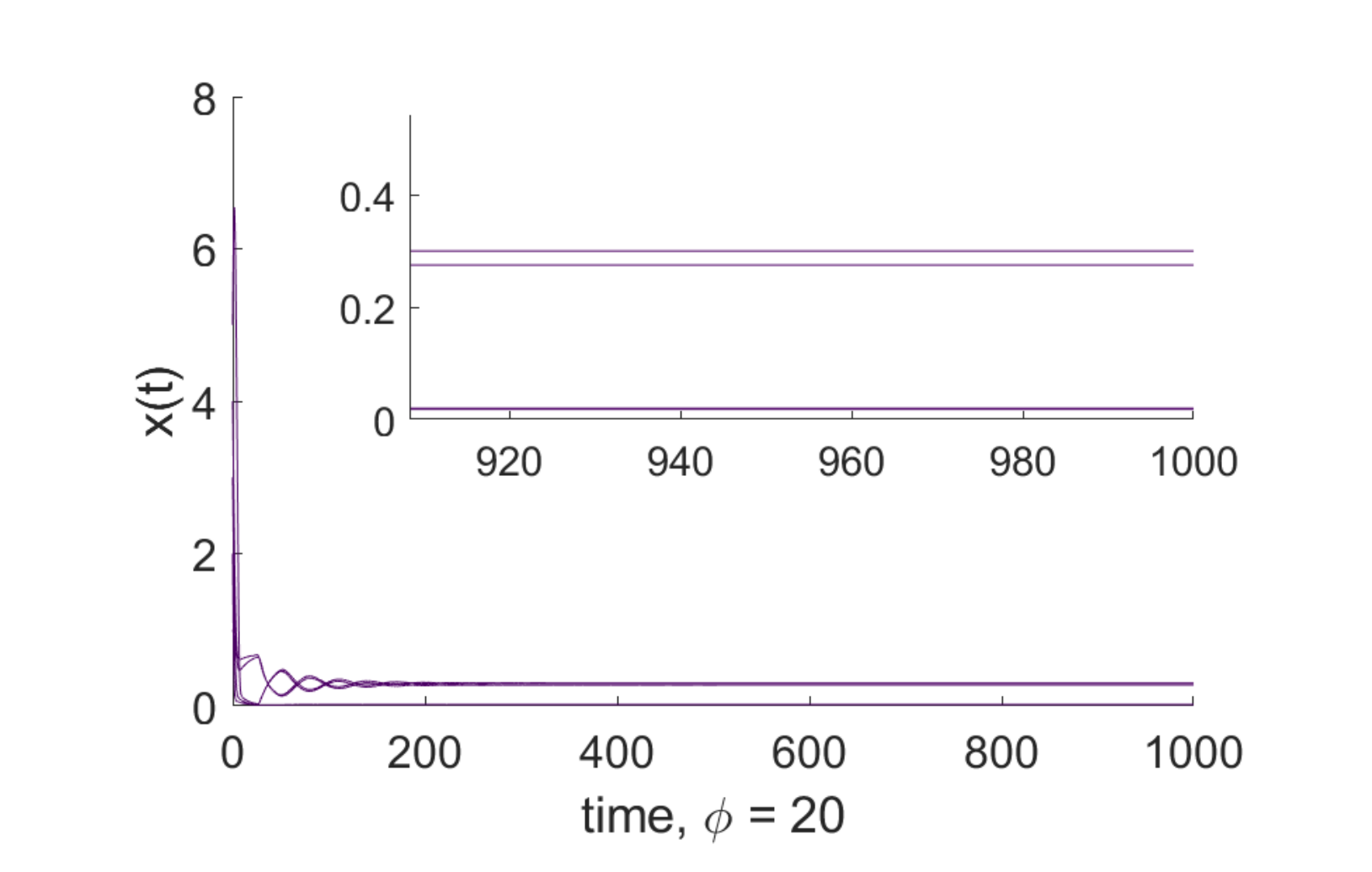}}
	\caption{\label{FigTraiett2} Trajectories with constant marginal tolls, for different values of  the delay $\phi$.} 
\end{figure}
From Figure~\ref{FigTraiett2} we can note that for all considered values of $\phi$ the trajectory $x$ converges to the equilibrium. This differs from what happens using the marginal cost tolls (see Figure~\ref{FigTraiett}) and highlights how time-delays affect marginal cost tolls more than their constant counterpart. The plot of the 1-norm, Figure \ref{normaNuovo2}, confirms the same trend, indeed after some initial oscillations, the 1-norm is the same for the different values of $\phi$.
\begin{figure}
	\centering	
	\includegraphics[scale=0.42]{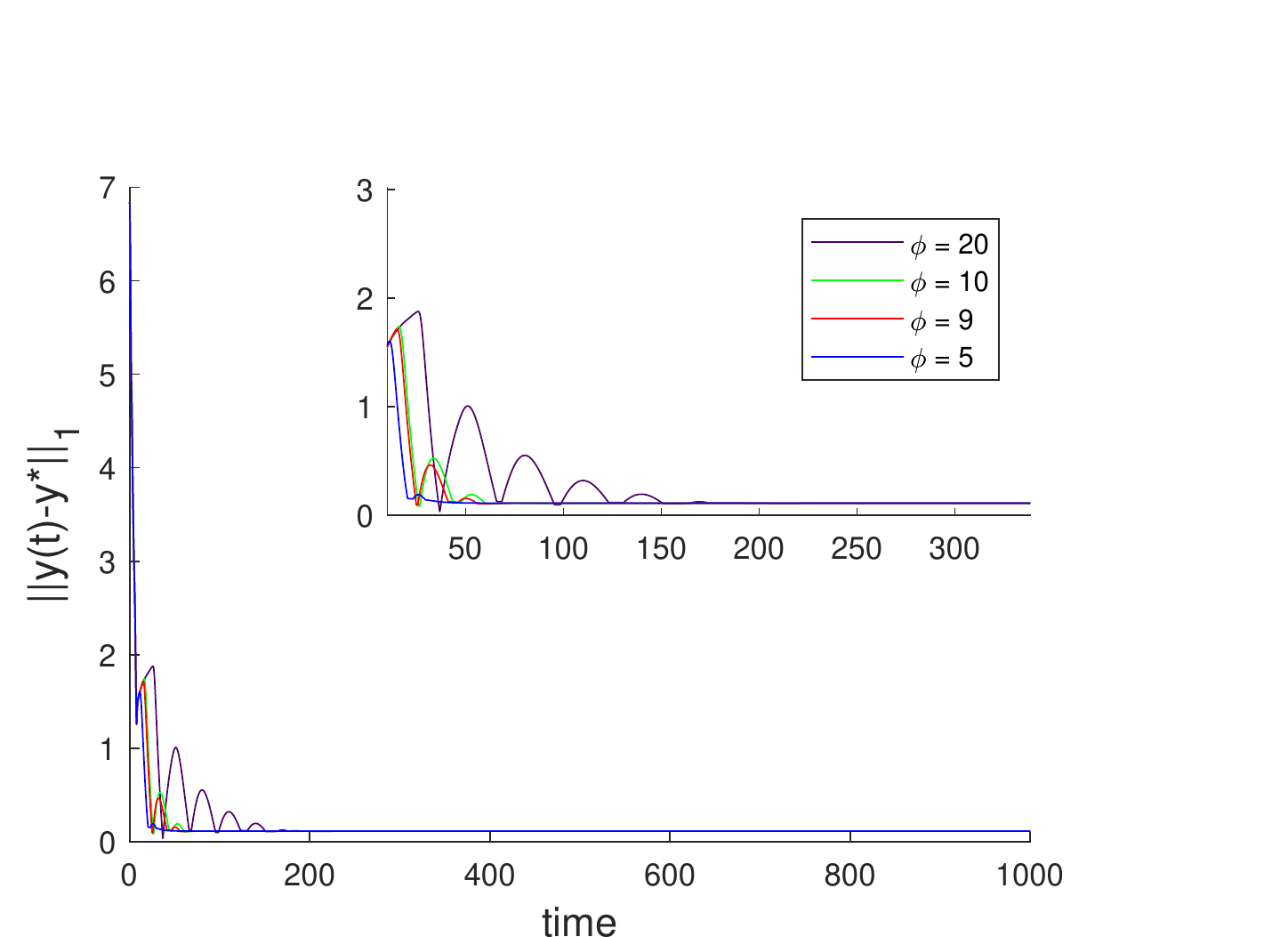}
	\caption{\label{normaNuovo2} Plot of $\Vert y(t) - y^* \Vert_1$ for different values of $\phi$.} 
\end{figure}

\section{Conclusion}\label{section6}
In this paper, we have studied the stability of multi-scale dynamical transportation networks with distributed dynamic feedback pricing. In particular, we have proved that, if the frequency of path preferences updates is sufficiently low, monotone decentralized \CorrR{flow-dependent dynamical tolls} make the network asymptotically approach a neighborhood of a generalized Wardrop equilibrium. For a particular class of dynamic feedback tolls, i.e., the marginal cost ones, we have proved that the stability is guaranteed to be around the social optimum equilibrium. 

Through numerical experiments, both asymptotic and transient performance of the system have been shown to be better with dynamic feedback marginal cost tolls than with constant ones. Finally, the impact of information delays has been investigated through numerical simulations, showing how such delays influence the stability and convergence of the network flow dynamics. In particular, our numerical simulations suggest that feedback marginal cost tolls might be more fragile to information delays that constant tolls. 

These numerical results motivate future research aimed at providing analytical estimates of the different convergence rates.
Moreover, it would also be worth analytically investigating the issue of robustness of feedback tolls to information delays and to consider anticipatory learning dynamics that incorporate derivative actions (c.f., e.g., \cite{ArslanShamma:2006}). 

\appendices
\section{Proof of Lemma \ref{lemma:convexity}}\label{proof:lemma-convexity}
The fact that the latency function $\tau_i(y)$ is twice continuously differentiable on $[0,C_i)$, strictly increasing, and such that $\tau_i(0) > 0$ directly follows from Assumption \ref{assumption:flow-function}. 

For a given $y$ in $[0,C_i)$, let $x=\varphi_i^{-1}(y)$, $a=\varphi_i'(x)$, and $b=\varphi_i''(x)$. Then,
$$\tau_i'(y)=\frac{\de}{\de y}\left(\frac{\varphi_i^{-1}(y)}{y}\right)=\frac{y/a-x}{y^2}=\frac{y-ax}{ay^2}\,,$$
thus proving \eqref{tau'}. 

We now prove that $y\mapsto y\tau_i(y)$ is strictly convex by computing its second derivative. For that, first notice that 
$$\frac{\de a}{\de y}=\frac{\de}{\de y}\varphi_i'(\varphi^{-1}(y))=\frac{\varphi_i''(x)}{\varphi_i'(x)}=\frac ba\,,$$
$$\frac{\de }{\de y}\left(y-ax\right)=1-\frac bax-a\frac1a=-\frac bax\,,$$
and
$$\frac{\de }{\de y}\left(ay^2\right)=\frac bay^2+2ya\,.$$ 
Then, 
\begin{equation*}
	\ba{rcl}
		\!\!\!(y\tau_i(y))^{''}\!\!\!&\!\!\!=\!\!\!& 2\tau_i'(y)+y\tau_i''(y)\\[8pt]
		&\!\!\!=\!\!\!&\ds	\frac{2(y-xa)}{ay^2}+y\frac{\de}{\de y}\left(\frac{y-xa}{ay^2}\right)	\\[8pt]
		&\!\!\!=\!\!\!&\ds	\frac{2(y-xa)}{ay^2}+\frac{\ds-bxy^2-(y-xa)\left(y^2\frac ba+2ya\right)}{a^2y^3}	\\[8pt]
		&=&-\ds\frac{b}{a^3}\,.
	\ea
\end{equation*}
Now, observe that Assumption \ref{assumption:flow-function} guarantees that $a>0$ and $b<0$. Hence, $(y\tau_i(y))^{''}>0$ and therefore $ y\tau_i(y)$ is strictly convex, thus completing the proof.
\qed 

\section{Proof of Proposition \ref{prop:Wardrop}}\label{proof:prop-Wardrop}
From Assumption \ref{assumption:flow-function} and the fact that the toll on a link is a non-decreasing function of the flow on that link only, it follows that the perceived cost function $\tau_i(y_i)+\omega_i(y_i)$ on link $i$ is continuous, strictly increasing, and grater than zero when $y_i=0$. The claim then follows as a direct application of Theorems 2.4 and 2.5 in \cite{Patriksson}. \qed

\begin{IEEEbiography}[{\includegraphics[width=1in,height=1.25in,clip,keepaspectratio]{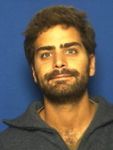}}]{Giacomo Como}
 is an Associate Professor at the Department of Mathematical Sciences, Politecnico di Torino, Italy, and at the Automatic Control Department of Lund University, Sweden. He received the B.Sc., M.S., and Ph.D. degrees in Applied Mathematics from Politecnico di Torino, in 2002, 2004, and 2008, respectively. He was a Visiting Assistant in Research at Yale University in 2006-2007 and a Postdoctoral Associate at the Laboratory for Information and Decision Systems, Massachusetts Institute of Technology, from 2008 to 2011. He currently serves as Associate Editor of IEEE-TCNS and IEEE-TNSE and as chair of the IEEE-CSS Technical Committee on Networks and Communications. He was the IPC chair of the IFAC Workshop NecSys'15 and a semiplenary speaker at the International Symposium MTNS'16. He is recipient of the 2015 George S. Axelby Outstanding Paper award. His research interests are in information, control, and network systems.  
\end{IEEEbiography}
\begin{IEEEbiography}[{\includegraphics[width=1in,height=1.25in,clip,keepaspectratio]{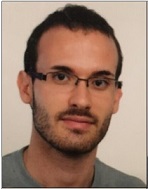}}]{Rosario Maggistro}
%
received the B.Sc. and M.S. in mathematics from University of Messina, Italy, in 2011 and 2013, respectively. In 2017, he received the Ph.D. degree in mathematics from University of Trento, Italy. He is a Postdoctoral Researcher at the Department of Management, Ca' Foscari Università di Venezia, Italy. In 2016, he was a Visiting Research Student at the Department of Automatic Control and Systems Engineering, Sheffield University (UK) and in 2017 he was a Visiting Research Fellow at the Department of Automatic Control, Lund University, Sweden. From 2017 to mid-2018 he was a Research Assistant at the Department of Mathematical Sciences, Politecnico di Torino, Italy. His current research interests include optimal control problems on network, mean field games and 
traffic/pedestrian flow models.
\end{IEEEbiography}
%
%
%
%
%
%
\end{document}